\documentclass{paper}
\pdfoutput=1

\usepackage{import}
\usepackage{quiver}
\usepackage{csquotes}

\newcommand{\juliette}[1]{}%{{\color{red} \sf $\spadesuit\spadesuit\spadesuit$ Juliette: [#1]}}
\newcommand{\lauren}[1]{}%{{\color{orange} \sf $\clubsuit$ Lauren: [#1]}}

\newcommand{\reg}{\operatorname{reg}}
\newcommand{\trunc}{\operatorname{trunc}}
\newcommand{\betti}{\operatorname{betti}}

\newcommand{\Coh}{\operatorname{Coh}}

\newcommand{\Om}{\Omega}

\newcommand{\shM}{\tilde M} % sheafy M
\newcommand{\cT}{\mathcal T} % tangent bundle
\newcommand{\cW}{\mathcal W} % bundle from 3.3
\newcommand{\cK}{\mathcal K} % resolution of the diagonal
\newcommand{\cL}{\mathcal L}
\newcommand{\cE}{\mathcal E}
\newcommand{\cF}{\mathcal F}
\newcommand{\cG}{\mathcal G}

\newcommand{\R}  {\mathbf R} % right derived
\renewcommand{\L}{\mathbf L} % left derived
\newcommand{\Gs}{\Gamma_{\!*}}

\newcommand{\cech}{\check{C}}
\newcommand{\Alt}{\bigwedge\nolimits}

\renewcommand{\aa}{\mathbf a}
\newcommand{\bb}{\mathbf b}
\newcommand{\cc}{\mathbf c}
\newcommand{\dd}{\mathbf d}
\newcommand{\ee}{\mathbf e}

\newcommand{\vv}{\mathbf v}
\newcommand{\ww}{\mathbf w}
\newcommand{\pp}{\mathbf p}
\newcommand{\nn}{\mathbf n}
\newcommand{\one}{\mathbf 1}
\newcommand{\zero}{\mathbf 0}

\newcommand{\PG}[1]{{\color{PineGreen}#1} }
\newcommand{\dao}[1]{{\color{DarkOrchid}#1} }

\title{Characterizing Multigraded Regularity and \\ Virtual Resolutions on Products of Projective Spaces}

\author{Juliette Bruce}
\address{Department of Mathematics, Dartmouth College, Hanover, NH}
\email{\href{mailto:juliette.bruce@dartmouth.edu}{juliette.bruce@dartmouth.edu}}
\urladdr{\url{https://www.juliettebruce.xyz}}

\author{Lauren Cranton Heller}
\address{Department of Mathematics, University of Nebraska, Lincoln, NE}
\email{\href{mailto:lheller2@unl.edu}{lheller2@unl.edu}}
\urladdr{\url{https://lcrantonh.github.io/}}

\author{Mahrud Sayrafi}
\address{School of Mathematics, University of Minnesota, Minneapolis, MN}
\email{\href{mailto:mahrud@umn.edu}{mahrud@umn.edu}}
\urladdr{\url{https://math.umn.edu/~mahrud/}}


\begin{document}

\begin{abstract}
  We explore the relationship between multigraded Castelnuovo--Mumford regularity, truncations, Betti numbers, and virtual resolutions on a product of projective spaces $X$\!.\linebreak After proving a uniqueness theorem for certain virtual resolutions, we show that the multigraded regularity region of a module $M$ is determined by the minimal graded free resolutions of the truncations $M_{\geq\dd}$ for $\dd\in\Pic X$\!. Further, by relating the minimal graded free resolutions of $M$ and $M_{\geq\dd}$ we provide a new bound on multigraded regularity of $M$ in terms of its Betti numbers. Using this characterization of regularity and this bound we also compute the multigraded Castelnuovo--Mumford regularity for a wide class of complete intersections in products of projective spaces.
\end{abstract}

\maketitle
%\tableofcontents

%%%%%%%%%%%%%%%%%%%%%%%%%%%%%%%%%%%%%%%%%%%%%%%%%%%%%%%%%%%%%%%%%%%%%%%%%%%%%%%%%%%%%

\section{Introduction}

Castelnuovo--Mumford regularity of coherent sheaves on a projective variety is a measure of complexity in terms of the vanishing of sheaf cohomology. Its geometric significance has been studied extensively for projective spaces \cite{Mumford1966}, abelian varieties \cite{PP03,PP04}, Grassmannians \cite{Chipalkatti00}, and smooth projective toric varieties \cite{MS04}, and it has been crucial in the construction of Hilbert and Picard schemes \cite{Kleiman71}. In many of these cases regularity is connected to minimal free resolutions and syzygies of graded modules \cite{Mumford70,BM91}.

Consider the case of a single projective space.  Let $S$ be the polynomial ring on $n+1$ variables over an algebraically closed field $\kk$ and $\m$ its maximal homogeneous ideal. A coherent sheaf $\cF$ on $\PPn=\Proj S$ is $d$-regular for $d\in\ZZ$ if % see \cite[Lec.~14]{Mumford1966}
\begin{enumerate}
\item\label{item:sheaf-coh} $\HH^i(\PPn\!, \cF(b)) = 0$ for all $i > 0$ and all $b \geq d-i$.
\end{enumerate}
The Castelnuovo--Mumford regularity of $\cF$ is then the minimum $d$ such that $\cF$ is $d$-regular. In \cite{EG84}, Eisenbud and Goto considered the analogous condition on the local cohomology of a finitely generated graded $S$-module $M$\!, proving the equivalence of the following:
\begin{enumerate} \setcounter{enumi}{1}
\item\label{item:local-coh} $\HH_\m^i(M)_b = 0$ for all $i \geq 0$ and all $b > d-i$;
\item\label{item:single-lin} the truncation $M_{\geq d}$ has a linear free resolution;
\item\label{item:single-betti} $\Tor_i^S(M, \kk)_b = 0$ for all $i \geq 0$ and all $b > d+i$.
\end{enumerate}
In particular, conditions \eqref{item:sheaf-coh} through \eqref{item:single-betti} are equivalent when $M = \bigoplus_p\HH^0(\PPn\!, \cF(p))$, the graded $S$-module corresponding to $\cF$, so that $\HH_\m^0(M) = \HH_\m^1(M) = 0$ (cf.\ \cite[Prop.~4.16]{Eisenbud2005}).

In \cite{MS04}, Maclagan and Smith introduced the notion of multigraded Castelnuovo--Mumford regularity for finitely generated $\Pic(X)$-graded modules over the Cox ring of a smooth projective toric variety $X$\!. In essence their definition is a generalization of conditions \eqref{item:sheaf-coh} and \eqref{item:local-coh}.\linebreak In this setting the multigraded regularity of a module is a subset of $\Pic X$ rather than a single integer. When $X = \PPn$ the minimum element of this region is the classical regularity.

In the multigraded case, translating the geometric definition of Maclagan and Smith into algebraic conditions like \eqref{item:single-lin} and \eqref{item:single-betti} above has been an open problem. In this direction, Maclagan--Smith and later Berkesch--Erman--Smith demonstrated connections between multigraded regularity and the existence of virtual resolutions with certain twists in \cite[Thm.~7.8]{MS04} and \cite[Thm.~2.9]{BES20}. In a more general setting, Botbol--Chardin sharpened the relationship between local cohomology and multigraded Betti numbers \cite[Thm.~4.14]{BC17}. More recently, Brown--Erman explored different notions of linearity for weighted projective spaces \cite{BE25} in relation to Green's $N_p$-conditions and Benson's weighted regularity \cite{Benson04}.

In this article we focus on the case when $X$ is a product of projective spaces and establish a tight relationship between multigraded regularity, truncations, Betti numbers, and virtual resolutions. Our main results strengthen and clarify previous work in a number of directions: \linebreak
First, we extend the equivalence of \eqref{item:local-coh} and \eqref{item:single-lin} by modifying the notion of a linear resolution. Second, we prove a uniqueness theorem for virtual resolutions considered in \cite[Thm.~2.9]{BES20} and use it to show that they are precisely the minimal free resolutions of truncated modules. Finally, we provide an effective method for determining whether a specific element $\dd\in\Pic X$ lies in $\reg(M)$ without a cohomology computation.

%%%%%%%%%%%%%%%%%%%%%%%%%%%%%%%%%%%%%%%%%%%%%%%%%%%%%%%%%%%%%%%%%%%%%%%%%%%%%%%%%%%%%

\subsection{Truncations and Multigraded Regularity}

The obvious way one might hope to generalize Eisenbud and Goto's result to products of projective spaces is false: the truncation $M_{\geq\dd}$ of a $\dd$-regular multigraded module $M$ can have nonlinear maps in its minimal free resolution (see Example~\ref{ex:not-linear}).  We show that under a mild saturation hypothesis, multigraded Castelnuovo--Mumford regularity is determined by a different linearity condition, which we call \emph{quasilinearity} (see Definition~\ref{def:quasilinear-resolution}).

Let $S$ be the $\ZZ^r$-graded Cox ring of $\PP\nn\coloneqq\PP{n_1}\times\cdots\times\PP{n_r}$ and let $B$ be the irrelevant ideal. The definition of quasilinearity is inspired by the criterion from \cite[Thm.~2.9]{BES20}. As an example, on a product of two projective spaces the following complex contains all allowed twists for a quasilinear complex generated in degree zero:
\vspace*{-0.6em}
\[\begin{tikzcd}[]
  0&\lar \PG{S}&\lar
  \begin{matrix}
    \PG{S(-1,\;\;0)} \\[-3pt]
    \oplus \\[-3pt]
    \PG{S(\;\;0,-1)}
  \end{matrix}
  \; \oplus \;
  \begin{matrix}
    \dao{S(-1,-1)}
  \end{matrix}
  &\lar
  \begin{matrix}
    \PG{S(-2,\;\;\;0)} \\[-3pt]
    \oplus  \\[-3pt]
    \PG{S(-1,-1)}  \\[-3pt]
    \oplus  \\[-3pt]
    \PG{S(\;\;\;0,-2)}\\[-3pt]
  \end{matrix}
  \; \oplus \;
  \begin{matrix}
    \dao{S(-2,-1)}\\[-3pt]
    \oplus \\[-3pt]
    \dao{S(-1,-2)}
  \end{matrix}
  & \lar \cdots.
\end{tikzcd}\]
Within each term, the summands in the left column (green) are linear syzygies while those in the right column (pink) are nonlinear syzygies. In general, for twists $-\bb$ appearing in the $i$-th step of a quasilinear complex, the sum of the positive components of $\bb-\dd-\one$ is at most $i-1$, where $\dd$ is the degree of the generators in the first term.

Our main theorem characterizes multigraded regularity of modules on products of projective spaces in terms of the Betti numbers of their truncations. Note that the condition $\HH_B^0(M)=0$ is automatic if $M$ is the module associated to a sheaf on $\PP\nn$\!.

\begin{theoremalpha}\label{main-thm-reg}
  Let $M$ be a finitely generated $\ZZ^r$-graded $S$-module with $\HH_B^0(M) = 0$.  Then $M$ is $\dd$-regular if and only if $M_{\geq\dd}$ has a quasilinear resolution $F_\b$ with $F_0$ generated in degree $\dd$.
\end{theoremalpha}

The proof of Theorem~\ref{main-thm-reg} is based in part on a \v{C}ech--Koszul spectral sequence that relates the Betti numbers of $M_{\geq\dd}$ to the terms of the Beilinson spectral sequence which computes the Fourier--Mukai transform of $\shM(\dd)$.  Precisely, if $M$ is $\dd$-regular and $\HH_B^0(M)=0$ we prove
\begin{align}\label{eq:magic-equality}
  \dim_\kk\Tor_j^S(M_{\geq\dd}, \kk)_{\dd+\aa} = h^{|\aa|-j}\big(\PP\nn\!, \shM(\dd)\otimes\Om_{\PP\nn}^\aa(\aa)\big) \quad \text{for } |\aa|\geq j\geq 0,
\end{align}
where $\Om_{\PP\nn}^\aa$ is a product of cotangent sheaves. The regularity of $M$ implies certain cohomological vanishing for $\shM\otimes\Om_{\PP\nn}^\aa(\aa)$, which, using \eqref{eq:magic-equality}, implies quasilinearity of the resolution of $M_{\geq\dd}$.  Conversely, building upon \cite[Thm.~2.9]{BES20}, a computation of $\HH_B^i(S)$ in Section~\ref{sec:qlin-to-reg} shows that the cokernel of a quasilinear resolution generated in degree $\dd$ is $\dd$-regular.  Thus we give a criterion for regularity in degree $\dd$ that is computable in practice.

Since free resolutions of $M_{\geq\dd}$ are virtual resolutions of $M$\!, the equality of Betti numbers from \eqref{eq:magic-equality} naturally suggests that the virtual resolutions exhibited in \cite[Thm.~2.9]{BES20} are precisely the minimal free resolutions of the truncations of $M$\!. We prove this to be true by establishing the following uniqueness theorem for virtual resolutions with terms in a full strong exceptional collection.

\begin{theoremalpha}\label{main-thm-uniqueness}
  Let $F_\b$ and $G_\b$ be virtual resolutions of an $S$-module. If $F_\b$ and $G_\b$ have no unit maps and their terms are direct sums of $S(-\aa)$ for $\zero\leq\aa\leq\nn$, then they are isomorphic.
\end{theoremalpha}

Our proof of Theorem~\ref{main-thm-uniqueness} reduces the problem to the uniqueness of quasi-isomorphic minimal projective complexes over certain graded associative algebras, using established techniques from the theory of derived categories in algebraic geometry and the representation theory of finite-dimensional algebras.

Finally, note that since a linear resolution is necessarily quasilinear, having a linear truncation at $\dd$ is strictly stronger than being $\dd$-regular. That is to say, when $\HH_B^0(M) = 0$:
\[
\parbox{12em}{\centering $M_{\geq\dd}$ has a linear resolution generated in degree $\dd$} \implies
\parbox{15em}{\centering $M_{\geq\dd}$ has a quasilinear resolution generated in degree $\dd$} \iff
\text{$M$ is $\dd$-regular.}
\]
Despite not fully characterizing multigraded regularity, having a linear resolution after truncation remains a useful condition. In Theorem~\ref{thm:sreg-iff-lin}, we use \eqref{eq:magic-equality} to get a cohomological characterization of when $M_{\geq\dd}$ has a linear resolution. Understanding the geometric implications of linearity on toric varieties is an active area of research \cite{BE24a,BE24b,BE25}.

%%%%%%%%%%%%%%%%%%%%%%%%%%%%%%%%%%%%%%%%%%%%%%%%%%%%%%%%%%%%%%%%%%%%%%%%%%%%%%%%%%%%%

\subsection{Betti Numbers and Multigraded Regularity}

Unlike in the case of a single projective space, the multigraded Betti numbers of a module $M$\!, that is $\beta_{i,\bb}(M)\coloneqq\Tor_i^S(M, \kk)_\bb$, do not determine its multigraded regularity. For instance, in Example~\ref{ex:same-betti} we construct two modules with the same multigraded Betti numbers but different multigraded regularities.  Hence the Betti numbers of $M$ also do not determine the Betti numbers of $M_{\geq\dd}$.  Still, we can intersect combinatorially defined regions $L_i(\bb)$ and $Q_i(\bb)$ (see Figure~\ref{fig:regions}) to specify a subset of the degrees $\dd\in\ZZ^r$ where $M_{\geq\dd}$ has a linear or quasilinear minimal free resolution generated in degree $\dd$.

\begin{figure}
  \newcommand{\makegrid}{
  \path[use as bounding box] (-2.45,-1.25) rectangle (4.45,4.25);
  \foreach \x in {-2,...,4}
  \foreach \y in {-1,...,4}
    { \fill[gray,fill=gray] (\x,\y) circle (1.5pt); }
  \draw[-,  semithick] (-2.5,0)--(4.5,0);
  \draw[-,  semithick] (0,-1.5)--(0,4.5);
}

\begin{tikzpicture}[scale=.42]
  \node at (1,5) (a) {\scriptsize $i=0$};
  \path[fill=PineGreen!45] (1,2)--(1,4)--(1,4)--(4,4)--(4,2)--(1,2);
  \makegrid

  \draw[->, ultra thick,PineGreen] (1,2)--(1,4);
  \draw[->, ultra thick,PineGreen] (1,2)--(4,2);
  \fill[TealBlue,fill=Black] (1,2) circle (6pt);
\end{tikzpicture}\quad\quad
%%%%%%%%%%%%%%%%%%%%%%%%%%%%%%%%%%%%
%%%%%%%%%%%%%%%%%%%%%%%%%%%%%%%%%%%%
%%%%%%%%%%%%%%%%%%%%%%%%%%%%%%%%%%%%
\begin{tikzpicture}[scale=.42]
  \node at (1,5) (a) {\scriptsize $i=1$};
  \path[fill=PineGreen!45] (0,2)--(0,4)--(4,4)--(4,1)--(1,1)--(1,2)--(0,2);
  \makegrid

  \draw[->, ultra thick,PineGreen] (0,2)--(0,4);
  \draw[-, cap=round,ultra thick,PineGreen] (0,2)--(1,2);
  \draw[-, cap=round,ultra thick,PineGreen] (1,1)--(1,2);
  \draw[->, ultra thick,PineGreen] (1,1)--(4,1);
  \fill[TealBlue,fill=PineGreen] (0,2) circle (6pt);
  \fill[TealBlue,fill=PineGreen] (1,1) circle (6pt);
  \fill[TealBlue,fill=Black] (1,2) circle (6pt);
\end{tikzpicture}\quad\quad
%%%%%%%%%%%%%%%%%%%%%%%%%%%%%%%%%%%%
%%%%%%%%%%%%%%%%%%%%%%%%%%%%%%%%%%%%
%%%%%%%%%%%%%%%%%%%%%%%%%%%%%%%%%%%%
\begin{tikzpicture}[scale=.42]
  \node at (1,5) (a) {\scriptsize $i=2$};
  \path[fill=PineGreen!45] (-1,2)--(-1,4)--(4,4)--(4,0)--(1,0)--(1,1)--(0,1)--(0,2)--(-1,2);
  \makegrid

  \draw[->, ultra thick,PineGreen] (-1,2)--(-1,4);
  \draw[-, cap=round,ultra thick,PineGreen] (-1,2)--(0,2);
  \draw[-, cap=round,ultra thick,PineGreen] (0,2)--(0,1);
  \draw[-, cap=round,ultra thick,PineGreen] (0,1)--(1,1);
  \draw[-, cap=round,ultra thick,PineGreen] (1,1)--(1,0);
  \draw[->, ultra thick,PineGreen] (1,0)--(4,0);
  \fill[TealBlue,fill=PineGreen] (-1,2) circle (6pt);
  \fill[TealBlue,fill=PineGreen] (0,1) circle (6pt);
  \fill[TealBlue,fill=PineGreen] (1,0) circle (6pt);
  \fill[TealBlue,fill=Black] (1,2) circle (6pt);
\end{tikzpicture}\quad\quad
%%%%%%%%%%%%%%%%%%%%%%%%%%%%%%%%%%%%
%%%%%%%%%%%%%%%%%%%%%%%%%%%%%%%%%%%%
%%%%%%%%%%%%%%%%%%%%%%%%%%%%%%%%%%%%
\begin{tikzpicture}[scale=.42]
  \node at (1,5) (a) {\scriptsize $i=3$};
  \path[fill=PineGreen!45] (-2,2)--(-2,4)--(4,4)--(4,-1)--(1,-1)--(1,0)--(0,0)--(0,1)--(-1,1)--(-1,2)--(-2,2);
  \makegrid

  \draw[->, ultra thick,PineGreen] (-2,2)--(-2,4);
  \draw[-, cap=round,ultra thick,PineGreen] (-2,2)--(-1,2)--(-1,1)--(0,1)--(0,0)--(1,0)--(1,-1);
  \draw[->, ultra thick,PineGreen] (1,-1)--(4,-1);
  \fill[TealBlue,fill=PineGreen] (-2,2) circle (6pt);
  \fill[TealBlue,fill=PineGreen] (-1,1) circle (6pt);
  \fill[TealBlue,fill=PineGreen] (0,0) circle (6pt);
  \fill[TealBlue,fill=PineGreen] (1,-1) circle (6pt);
  \fill[TealBlue,fill=Black] (1,2) circle (6pt);
\end{tikzpicture}
\linebreak
\linebreak
\hspace*{0cm}
%%%%%%%%%%%%%%%%%%%%%%%%%%%%%%%%%%%%
%%%%%%%%%%%%%%%%%%%%%%%%%%%%%%%%%%%%
%%%%%%%%%%%%%%%%%%%%%%%%%%%%%%%%%%%%
%%%%%%%%%%%%%%%%%%%%%%%%%%%%%%%%%%%%
%%%%%%%%%%%%%%%%%%%%%%%%%%%%%%%%%%%%
%%%%%%%%%%%%%%%%%%%%%%%%%%%%%%%%%%%%
%%%%%%%%%%%%%%%%%%%%%%%%%%%%%%%%%%%%
%%%%%%%%%%%%%%%%%%%%%%%%%%%%%%%%%%%%
%%%%%%%%%%%%%%%%%%%%%%%%%%%%%%%%%%%%
\begin{tikzpicture}[scale=.42]
  \path[fill=DarkOrchid!50] (1,2)--(1,4)--(4,4)--(4,2)--(1,2);
  \makegrid

  \draw[->, ultra thick,DarkOrchid] (1,2)--(1,4);
  \draw[->, ultra thick,DarkOrchid] (1,2)--(4,2);
  \fill[TealBlue,fill=DarkOrchid] (1,2) circle (6pt);
  \fill[TealBlue,fill=Black] (1,2) circle (6pt);
\end{tikzpicture}\quad\quad
%%%%%%%%%%%%%%%%%%%%%%%%%%%%%%%%%%%%
%%%%%%%%%%%%%%%%%%%%%%%%%%%%%%%%%%%%
%%%%%%%%%%%%%%%%%%%%%%%%%%%%%%%%%%%%
\begin{tikzpicture}[scale=.42]
  \path[fill=DarkOrchid!50] (0,1)--(0,4)--(4,4)--(4,1)--(0,1);
  \makegrid

  \draw[->, ultra thick,DarkOrchid] (0,1)--(0,4);
  \draw[->, ultra thick,DarkOrchid] (0,1)--(4,1);
  \fill[TealBlue,fill=DarkOrchid] (0,1) circle (6pt);
  \fill[TealBlue,fill=Black] (1,2) circle (6pt);
\end{tikzpicture}\quad\quad
%%%%%%%%%%%%%%%%%%%%%%%%%%%%%%%%%%%%
%%%%%%%%%%%%%%%%%%%%%%%%%%%%%%%%%%%%
%%%%%%%%%%%%%%%%%%%%%%%%%%%%%%%%%%%%
\begin{tikzpicture}[scale=.42]
  \path[fill=DarkOrchid!50] (-1,1)--(-1,4)--(4,4)--(4,0)--(0,0)--(0,1)--(-1,1);
  \makegrid

  \draw[->, ultra thick,DarkOrchid] (-1,1)--(-1,4);
  \draw[-, cap=round,ultra thick,DarkOrchid] (-1,1)--(0,1);
  \draw[-, cap=round,ultra thick,DarkOrchid] (0,1)--(0,0);
  \draw[->, ultra thick,DarkOrchid] (0,0)--(4,0);
  \fill[TealBlue,fill=DarkOrchid] (-1,1) circle (6pt);
  \fill[TealBlue,fill=DarkOrchid] (0,0) circle (6pt);
  \fill[TealBlue,fill=Black] (1,2) circle (6pt);
\end{tikzpicture}\quad\quad
%%%%%%%%%%%%%%%%%%%%%%%%%%%%%%%%%%%%
%%%%%%%%%%%%%%%%%%%%%%%%%%%%%%%%%%%%
%%%%%%%%%%%%%%%%%%%%%%%%%%%%%%%%%%%%
\begin{tikzpicture}[scale=.42]
  \path[fill=DarkOrchid!50] (-2,1)--(-2,4)--(4,4)--(4,-1)--(0,-1)--(0,0)--(-1,0)--(-1,1)--(-2,1);
  \makegrid

  \draw[->, ultra thick,DarkOrchid] (-2,1)--(-2,4);
  \draw[-, cap=round,ultra thick,DarkOrchid] (-2,1)--(-1,1)--(-1,0)--(0,0)--(0,-1);
  \draw[->, ultra thick,DarkOrchid] (0,-1)--(4,-1);
  \fill[TealBlue,fill=DarkOrchid] (-2,1) circle (6pt);
  \fill[TealBlue,fill=DarkOrchid] (-1,0) circle (6pt);
  \fill[TealBlue,fill=DarkOrchid] (0,-1) circle (6pt);
  \fill[TealBlue,fill=Black] (1,2) circle (6pt);
\end{tikzpicture}
	\caption{The top row shows the regions $L_i(1,2)$ in green, and the bottom row $Q_i(1,2)$ in pink, for $i=0,1,2,3$, from left to right, as defined in Section~\ref{sec:regularity}.}\label{fig:regions}
\end{figure}

\begin{theoremalpha}\label{main-thm-betti}
  Let $M$ be a finitely generated $\ZZ^r$-graded $S$-module.
  \begin{enumerate}
  \item If $\displaystyle \dd\in\bigcap_{i\in\NN} \bigcap_{\bb\in\beta_i(M)} Q_i(\bb)$ then $M_{\geq \dd}$ has a quasilinear resolution generated in degree $\dd$.  \label{item:reg-bound}
  \item If $\displaystyle \dd\in\bigcap_{i\in\NN} \bigcap_{\bb\in\beta_i(M)} L_i(\bb)$ then $M_{\geq \dd}$ has a linear resolution generated in degree $\dd$.
  \end{enumerate}
  Here we set $\beta_i(M)\coloneqq\{\bb\in\ZZ^r \;|\; \Tor_i^S(M, \kk)_\bb\neq 0\}$ to be the degree support of $\Tor_i^S(M, \kk)$. % the set of degrees where the Betti numbers of $M$ are nonzero.
\end{theoremalpha}

On a single projective space the regions $L_i(\bb)$ and $Q_i(\bb)$ coincide, so we recover condition \eqref{item:single-betti} of Eisenbud--Goto.  Our proof of Theorem~\ref{main-thm-betti} is based on the observation that we can construct a possibly nonminimal free resolution of $M_{\geq\dd}$ from the truncations of the terms in the minimal free resolution of $M$\!.

A number of inner\footnote{We use the terms inner and outer bound since in general there is no total ordering on $\reg(M)$ when $\Pic X\neq\ZZ$. For a single projective space an inner bound corresponds to an upper bound and an outer bound to a lower bound.} bounds on the multigraded regularity of a module in terms of its Betti numbers exist in the literature. For example, \cite[Cor.~7.3]{MS04} used a local cohomology long exact sequence argument to deduce such a bound. These methods were extended in \cite[Thm.~4.14]{BC17} using a local cohomology spectral sequence. Our bound in Theorem~\ref{main-thm-betti} is generally larger and thus closer to the actual regularity than these results.

Moreover, Theorem~\ref{main-thm-betti} is sharp in a number of examples.  For instance, we use Theorem~\ref{main-thm-reg} to show that the containment in \eqref{item:reg-bound} is equal to the regularity for all saturated \emph{ample} complete intersections, meaning those determined by ample hypersurfaces. % Parameterized by the degrees of their generators, almost all complete intersections are ample.

\begin{theoremalpha}\label{main-thm-ci}
	Suppose $\langle f_1,\ldots,f_c\rangle\subset B$ is a saturated complete intersection of codimension $c$ in $S$, so the affine subvariety defined by it contains the irrelevant locus $V(B)$. Then
	\[ \reg\frac{S}{\langle f_1,\ldots,f_c\rangle} = Q_c\left(\sum_{i=1}^c\deg f_i\right). \]
%	\[ \reg(S/I) = Q_c\left(\sum_{i=1}^c\deg f_i\right). \]
\end{theoremalpha}

Note that on a product of projective spaces the intermediate cohomology of a complete intersection does not necessarily vanish. Even the local cohomology of a hypersurface in a product of projective spaces is not determined by its degree \cite[Sec.~4.5]{BC17}. Thus computing the multigraded regularity of complete intersections on products of projective spaces is more complicated than in the case of a single projective space.

Our work highlights a pattern in the literature: many classical results connecting the geometry of projective varieties and homological algebra do not have straightforward generalizations to other toric varieties. When such generalizations do exist they often highlight surprising and subtle connections between geometry and algebra. See for instance other work on multigraded regularity \cite{MS04,HW04,SVT04,SVTW06,Ha07,CMR07,BC17,CN20}, Tate resolutions \cite{EES15,BE26}, virtual resolutions \cite{BES20,Loper21,Yang21,BKLY21,HNVT22}, and syzygies \cite{HSS06,HS07,Hering10,Bruce22,Bruce24}.

%%%%%%%%%%%%%%%%%%%%%%%%%%%%%%%%%%%%%%%%%%%%%%%%%%%%%%%%%%%%%%%%%%%%%%%%%%%%%%%%%%%%%

%\tableofcontents

\subsection*{Outline}

The organization of the paper is as follows:
Section~\ref{sec:background} gathers background results and fixes our notation.
Section~\ref{sec:virtual-res} proves the uniqueness of virtual resolutions consisting only of certain twists, in Theorem~\ref{thm:uniqueness}.
Section~\ref{sec:regularity-criterion} proves our main result, characterizing multigraded regularity in terms of quasilinearity of resolutions of truncations, in Theorem~\ref{thm:quasilinear-conjecture}. %One direction of the relationship between regularity and quasilinearity appears in Section~\ref{sec:reg-to-qlin} (Theorem~\ref{thm:reg-to-qlin}) and the other in Section~\ref{sec:qlin-to-reg} (Theorem~\ref{thm:qlin-to-reg}).
Section~\ref{sec:betti} relates multigraded regularity to Betti numbers, by describing the Betti numbers of truncations in Theorems~\ref{thm:betti-truncations} and \ref{thm:betti-regularity} and computing the multigraded regularity of a class of complete intersections in Theorem~\ref{thm:ci-regularity}.
Section~\ref{sec:linear-trun} sharpens our theorems in the case of linear truncations.
Finally, Section~\ref{sec:regularity-regions} summarizes our results about the regions defined by truncations, Betti numbers, and multigraded regularity.

%%%%%%%%%%%%%%%%%%%%%%%%%%%%%%%%%%%%%%%%%%%%%%%%%%%%%%%%%%%%%%%%%%%%%%%%%%%%%%%%%%%%%

\subsection*{Acknowledgments}

We thank Christine Berkesch, Daniel Erman, and Gregory G.~Smith for conversations which refined our understanding of the relationship between virtual resolutions and regularity, and David Eisenbud for drawing our attention to linear truncations. We are also grateful to Daniel Erman, Michael Brown, David Favero, and Peter Webb for their help in simplifying the proofs of Proposition~\ref{prop:free-monad} and Theorem~\ref{thm:uniqueness}, Monica Lewis for pointing out an improvement to Theorem~\ref{thm:ci-regularity}, as well as Michael Loper, Navid Nemati, Ben Briggs, and Felix K\"ung for helpful conversations. We thank the referees for their suggestions, which improved the presentation of our results. The computer software \textit{Macaulay2} \cite{M2} was vital in shaping our early conjectures.

The first author is grateful for the support of the Mathematical Sciences Research Institute in Berkeley, California, where she was in residence for the 2020--2021 academic year. The first author was partially supported by the National Science Foundation under Award Nos. DMS-1440140 and NSF MSPRF DMS-2002239. The third author was partially supported by the NSF grants DMS-1745638, DMS-2001101, and DMS-1502209.

%%%%%%%%%%%%%%%%%%%%%%%%%%%%%%%%%%%%%%%%%%%%%%%%%%%%%%%%%%%%%%%%%%%%%%%%%%%%%%%%%%%%%
%\pagebreak[4]

\section{Notation and Background}\label{sec:background}

Throughout we denote the natural numbers by $\NN=\{0,1,2,\ldots\}$.  When referring to vectors in $\ZZ^r$ we use a bold font.  Given a vector $\vv = (v_1,\ldots,v_r)\in\ZZ^r$ we denote the sum $v_1+\cdots+v_r$ by $|\vv|$.  For $\vv,\ww\in\ZZ^r$ we write $\vv\leq\ww$ when $v_i\leq w_i$ for all $i$, and use $\max\{\vv,\ww\}$ to denote the vector whose $i$-th component is $\max\{v_i,w_i\}$.  We reserve $\ee_1,\ldots,\ee_r$ for the standard basis of $\ZZ^r$ and for brevity we write $\one$ for $(1,1,\ldots,1)\in\ZZ^r$ and $\zero$ for $(0,0,\ldots,0)\in\ZZ^r$\!.

Fix a Picard rank $r\in \NN$ and dimension vector $\nn = (n_1,\ldots,n_r)\in\NN^r$\!. We denote by $\PP\nn$ the product $\PP{n_1}\times\cdots\times\PP{n_r}$ of $r$ projective spaces over a field $\kk$. Given $\bb\in\ZZ^r$ we let
\[ \OO_{\PP\nn}(\bb) \coloneqq \pi_1^*\OO_{\PP{n_1}}(b_1)\otimes\cdots\otimes\pi_r^*\OO_{\PP{n_r}}(b_r) \]
where $\pi_i$ is the projection of $\PP\nn$ to $\PP{n_i}$\!. This gives an isomorphism $\Pic\PP\nn\cong\ZZ^r$\!, which we use implicitly throughout.

Let $S$ be the $\ZZ^r$-graded Cox ring of $\PP\nn$\!, which is isomorphic to the polynomial ring $\kk[x_{i,j} \; | \; 1 \leq i \leq r, \; 0 \leq j \leq n_i]$ with $\deg(x_{i,j})=\ee_i$. Further, let $B=\bigcap_{i=1}^r \langle x_{i,0},x_{i,1},\ldots,x_{i,n_i}\rangle\subset S$ be the irrelevant ideal. For a description of the Cox ring and the relationship between coherent $\OO_{\PP\nn}$-modules and $\ZZ^r$-graded $S$-modules, see \cites{Cox95,CLS2011}. In particular, the twisted global sections functor $\Gs$ given by $\cF\mapsto \bigoplus_{\pp\in\ZZ^r} \HH^0(\PP\nn\!, \cF(\pp))$ takes coherent sheaves on $\PP\nn$ to $S$-modules.  Given a $\ZZ^r$-graded $S$-module $M$\!, we write $\Tor_i^S(M,\kk)_\bb$ for the degree~$\bb$ part of the module $\Tor_i^S(M,\kk)$ and use $\beta_i(M)\coloneqq\{\bb\in\ZZ^r \;|\; \Tor_i^S(M, \kk)_\bb\neq 0\}$ to denote the set of multidegrees of $i$-th syzygies of $M$\!.

%%%%%%%%%%%%%%%%%%%%%%%%%%%%%%%%%%%%%%%%%%%%%%%%%%%%%%%%%%%%%%%%%%%%%%%%%%%%%%%%%%%%%

\subsection{Multigraded Regularity}\label{sec:regularity}

In order to streamline our definitions of regions inside the Picard group of $\PP\nn$\!, we introduce the following subsets of $\ZZ^r$: for $\dd\in\ZZ^r$ and $i\in\NN$ let
\begin{align*}
  L_i(\dd) &\coloneqq \bigcup_{|\lambda|=i}\left(\dd-\lambda_1\ee_1-\cdots-\lambda_r\ee_r+\NN^r\right)
  \quad\text{for $\lambda_1,\ldots,\lambda_r\in\NN$} \\
  Q_i(\dd) &\coloneqq L_{i-1}(\dd-\one) \quad \text{for $i>0$}\quad\text{and} \quad Q_0(\dd)=\dd+ \NN^r.
\end{align*}
Note that for fixed $\dd\in\ZZ^r$ we have $L_i(\dd)\subseteq Q_i(\dd)$ for all $i$.

\begin{example}
	When $r=2$ the regions $L_i(\dd)$ and $Q_i(\dd)$ can be visualized as in Figure~\ref{fig:regions}. For $i>1$ they are shaped like staircases with $i+1$ and $i$ ``corners,'' respectively; in other words $L_i(\dd)$ contains $i+1$ minimal elements and $Q_i(\dd)$ contains $i$.
\end{example}

\begin{remark}\label{rem:describeL}
	An alternate description of $L_i(\dd)$ will also be useful: it is the set of $\bb\in\ZZ^r$ so that the sum of the positive components of $\dd-\bb$ is at most $i$.  (This ensures that we can distribute the $\lambda_j$ so that $\bb+\sum_j\lambda_j\ee_j\geq \dd$.)
\end{remark}

With this notation in hand we can recall the definition of multigraded regularity.

\begin{definition}\label{def:regular}\cite[Def.~1.1]{MS04}
	Let $M$ be a finitely generated $\ZZ^r$-graded $S$-module. We say $M$ is \emph{$\dd$-regular} for $\dd\in\ZZ^r$ if the following hold:
	\begin{enumerate}
	\item  $\HH_B^0(M)_\pp=0$ for all $\pp\in\bigcup_{1\leq j\leq r}\left(\dd+\ee_j+\NN^r\right)$,
	\item  $\HH_B^i(M)_\pp=0$ for all $i>0$ and $\pp\in L_{i-1}(\dd)$. \label{item:reg-def}
	\end{enumerate}
	The \emph{multigraded Castelnuovo--Mumford regularity} of $M$ is then the set
	\[ \reg(M)\coloneqq \left\{\dd\in\ZZ^r \;\; \big| \;\; \text{$M$ is $\dd$-regular} \right\} \subseteq \Pic\PP\nn\cong\ZZ^r. \]
\end{definition}

It follows directly from the definition that if $M$ is $\dd$-regular, then $M$ is $\dd'$-regular for all $\dd'\geq\dd$.  For other properties of multigraded regularity, such as $\zero$-regularity of $S$, see \cite{MS04}.

\begin{remark}
  Several alternate notions of Castelnuovo--Mumford regularity for the multigraded setting exist in the literature. The initial extension was introduced by Hoffman and Wang for a product of two projective spaces \cite{HW04}. Following Maclagan and Smith's definition, Botbol and Chardin gave a more general definition working over an arbitrary base ring \cite{BC17}. Recently, in their work on Tate resolutions on toric varieties, Brown and Erman introduced a modified notion of multigraded regularity for a weighted projective space, which they then extended to other toric varieties \cite[\S6.1]{BE26}.
\end{remark}

% For arXiv sleuths:
%\begin{definition}
%	We say $M$ is \emph{strongly $\dd$-regular} \lauren{can we get rid of this name?} for $\dd\in\ZZ^r$ if:
%	\begin{enumerate}
%		\item  \lauren{what is the right condition on H0?}
%		\item  $\HH_B^i(M)_{\pp}=0$ for all $i\geq 0$ and $\pp\in Q_{i-1}(\dd)$
%	\end{enumerate}
%\end{definition}

%%%%%%%%%%%%%%%%%%%%%%%%%%%%%%%%%%%%%%%%%%%%%%%%%%%%%%%%%%%%%%%%%%%%%%%%%%%%%%%%%%%%%

\subsection{Truncations and Local Cohomology}\label{sec:truncations-local-coh}

In this section we collect facts about truncations and local cohomology that will be used repeatedly. As in the case of a single projective space, the truncation of a graded module on a product of projective spaces at multidegree $\dd$ contains all elements of degree at least $\dd$.

\begin{definition}
	For $\dd\in\ZZ^r$ and $M$ a $\ZZ^r$-graded $S$-module, the \emph{truncation} of $M$ at $\dd$ is the $\ZZ^{r}$-graded $S$-submodule $M_{\geq\dd} \coloneqq \bigoplus_{\dd'\geq\dd} M_{\dd'}$.
\end{definition}

Immediate from the definition is the following lemma.

\begin{lemma}\label{lem:truncation-functor}
The truncation map $M\mapsto M_{\geq \dd}$ is an exact functor of $\ZZ^r$-graded $S$-modules.
\end{lemma}

\begin{remark}
Since truncation is exact, if $F_{\b}$ is a graded free resolution of a module $M$ then the term by term truncation $(F_{\b})_{\geq \dd}$ is a resolution of $M_{\geq \dd}$. However, in general the truncation of a free module is not free, so $(F_{\b})_{\geq \dd}$ is generally not a free resolution of $M_{\geq \dd}$.
\end{remark}

We denote by $\HH_B^p(M)$ the $p$-th local cohomology of $M$ supported at the irrelevant ideal $B$. For $p>0$ and $\aa\in\ZZ^r$ there exist natural isomorphisms
\[ \HH^p\big(\PP\nn\!, \shM(\aa)\big) \cong \HH_B^{p+1}(M)_\aa, \]
and for $p=0$ there is a $\ZZ^r$-graded exact sequence
\begin{equation}\label{eq:gamma-star}
	\begin{tikzcd}
	0 \rar& \HH_B^0(M) \rar& M \rar& \Gs(\shM) \rar& \HH_B^1(M) \rar& 0.
	\end{tikzcd}
\end{equation}

An important tool for computing local cohomology is the local \v{C}ech complex
\[\begin{tikzcd}
  \cech^\b(B,M)\colon 0 \rar& M \rar& \bigoplus M[g_{i}^{-1}] \rar& \bigoplus M[g_{i}^{-1},g_{j}^{-1}] \rar& \cdots
\end{tikzcd}\]
where the $g_i$ range over the generators of $B$. We index the local \v{C}ech complex so that the summands of $\cech^p(B,M)$ are localizations of $M$ at $p$ distinct generators of $B$. Then we have
\[ \HH_B^p(M) \cong \HH^p(\cech^\b(B,M)). \]
See \cite{24hours} and \cite[\S9]{CLS2011} for more details.

Note that inverting a generator of $B$ inverts a variable from each factor of $\PP\nn$\!, so the distinguished open sets corresponding to the generators of $B$ form an affine cover $\mathfrak U_B$ of $\PP\nn$\!.  Denote by $\cech^\b(\mathfrak U_B, \cF)$ the \v{C}ech complex of a sheaf $\cF$ with respect to $\mathfrak U_B$:
\[\begin{tikzcd}
  \cech^\b(\mathfrak U_B,\cF)\colon 0 \rar& \bigoplus \cF|_{\mathfrak U_i} \rar& \bigoplus \cF|_{\mathfrak U_i\cap\mathfrak U_j} \rar& \cdots.
\end{tikzcd}\]

\begin{lemma}\label{lem:cech-horseshoe}
	Given a complex of graded $S$-modules \( L\to M\to N \) such that \( \tilde L\to\tilde M\to\tilde N \) is exact, the complex \( \cech^p(B, L) \to \cech^p(B,M) \to \cech^p(B,N) \) is exact for each $p\geq 0$.
\end{lemma}
\begin{proof}
	Fix $p$.  Then $\cech^p(B,L)\to\cech^p(B,M)\to\cech^p(B,N)$ splits as a direct sum of complexes
	\[L[g_1^{-1},\ldots,g_p^{-1}]\to M[g_1^{-1},\ldots,g_p^{-1}]\to N[g_1^{-1},\ldots,g_p^{-1}]\]
	each of which can be obtained by applying $\Gamma(U,-)$ to $\tilde L\to\tilde M\to\tilde N$, where $U$ is the complement of $V(g_1,\ldots,g_p)$.  Since $U$ is affine they are exact.
\end{proof}

Since $M/M_{\geq\dd}$ is annihilated by a power of $B$, a module $M$ and its truncation define the same sheaf on $\PP\nn$\!. In particular $\HH_B^p(M) = \HH_B^p(M_{\geq\dd})$ for $p\geq 2$.  The long exact sequence of local cohomology applied to $0\to M_{\geq\dd}\to M\to M/M_{\geq\dd}\to 0$ gives
\[\begin{tikzcd}
0 \rar& \HH_B^0(M_{\geq\dd}) \rar& \HH_B^0(M) \rar& M/M_{\geq\dd} \rar& \HH_B^1(M_{\geq\dd}) \rar& \HH_B^1(M) \rar& 0.
\end{tikzcd}\]
Hence $\HH_B^0(M) = 0$ implies $\HH_B^0(M_{\geq\dd}) = 0$.  Since $M/M_{\geq\dd}$ is zero in degrees larger than $\dd$ we also have $\HH_B^1(M_{\geq\dd})_{\geq\dd} = \HH_B^1(M)_{\geq\dd}$. An immediate consequence is the following lemma, which we will use repeatedly to reduce to the case when $\dd=\zero$.

\begin{lemma}\label{lem:truncation-regularity}
A $\ZZ^{r}$-graded $S$-module $M$ is $\dd$-regular if and only if $M_{\geq\dd}$ is $\dd$-regular.
\end{lemma}

%%%%%%%%%%%%%%%%%%%%%%%%%%%%%%%%%%%%%%%%%%%%%%%%%%%%%%%%%%%%%%%%%%%%%%%%%%%%%%%%%%%%%

\subsection{Koszul Complexes and Cotangent Sheaves}\label{sec:Koszul-cotangent}

For each factor $\PP{n_i}$ of $\PP\nn$\!, the Koszul complex on the variables of $S_i = \Cox\PP{n_i}$ is a resolution of $\kk$:
\begin{align}\label{eq:Koszul-on-factors}
	K^i_\b\colon 0 \gets S_i \gets S_i^{n_i+1}(-1)\gets\Alt^2 \left[S_i^{n_i+1}(-1)\right]\gets\cdots\gets\Alt^{n_i+1} \left[S_i^{n_i+1}(-1)\right]\gets 0.
\end{align}
The Koszul complex $K_\b$ on the variables of $S$ is the tensor product of the complexes $\pi_i^*K_\b^i$.

For $1\leq a\leq n_i$ let $\hat\Om_{\PP{n_i}}^a$ be the kernel of $\Alt^{a-1}\left[S_i^{n_i+1}(-1)\right]\gets\Alt^a\left[S_i^{n_i+1}(-1)\right]$ and let $\Om_{\PP{n_i}}^a$ denote its sheafification.  The minimal free resolution of $\hat\Om_{\PP{n_i}}^a$ then consists of the terms of $K_\b^i$ with homological index greater than $a$.  Write $\hat\Om_{\PP{n_i}}^0$ for the kernel of $\kk\gets S_i$ (so that $\Om_{\PP{n_i}}^0=\OO_{\PP{n_i}}$) and take $\hat\Om_{\PP{n_i}}^a$ to be 0 if $a\notin\{0,\ldots,n_i\}$.  For $\aa\in\ZZ^r$ with $\zero\leq\aa\leq\nn$ define
\[ \Om_{\PP\nn}^\aa \coloneqq\pi_1^*\Om_{\PP{n_1}}^{a_1}\otimes\cdots\otimes\pi_r^*\Om_{\PP{n_r}}^{a_r} \]
and write $\hat\Om_{\PP\nn}^\aa$ for the analogous tensor product of the modules $\hat\Om_{\PP{n_i}}^{a_i}$.

Given a free complex $F_\b$ and a multidegree $\aa\in\ZZ^r$\!, denote by $F_\b^{\leq\aa}$ the subcomplex of $F_\b$ consisting of free summands generated in degrees at most $\aa$.

\begin{lemma}\label{lem:finite-length-homology}
	Fix $\aa\in\ZZ^r$ and let $K_\b$ be the Koszul complex on the variables of $S$.  The subcomplex $K_\b^{\leq\aa}$ is equal to $K_\b$ in degrees $\leq\aa$, and its sheafification is exact except at homological index $|\aa|$, where it has homology $\Om_{\PP\nn}^\aa$.
\end{lemma}
\begin{proof}
	The first statement follows from the fact that the terms appearing in $K_\b$ but not $K_\b^{\leq\aa}$ have no elements in degrees $\leq\aa$.

	Note that $K_\b^{\leq\aa}$ is a tensor product of pullbacks of subcomplexes of the $K^i_\b$ in \eqref{eq:Koszul-on-factors}:
	\[ K_\b^{\leq\aa}=\pi_1^*(K_\b^1)^{\leq a_1}\otimes\cdots\otimes\pi_r^*(K_\b^r)^{\leq a_r}. \]
	After sheafification, each complex $\pi_i^*(K_\b^i)^{\leq a_i}$ is exact away from its kernel $\pi_i^*\Om_{\PP{n_i}}^{a_i}$, which appears at homological index $a_i$.  Thus $\tilde K_\b^{\leq\aa}$ has homology $\Om_{\PP\nn}^\aa$, appearing in index $|\aa|$.
\end{proof}

%%%%%%%%%%%%%%%%%%%%%%%%%%%%%%%%%%%%%%%%%%%%%%%%%%%%%%%%%%%%%%%%%%%%%%%%%%%%%%%%%%%%%

\section{Uniqueness of Virtual Resolutions}\label{sec:virtual-res}

In this section we characterize a class of \emph{virtual resolutions}, as defined by Berkesch, Erman, and Smith, for graded modules over the Cox ring $S$ of a product of projective spaces $\PP\nn$\!,\linebreak which consist only of certain twists depending on the dimension vector $\nn$. We will prove in Theorem~\ref{thm:uniqueness} that such virtual resolutions are unique up to isomorphism of complexes. In particular, Proposition~\ref{prop:free-monad} uses the Fourier--Mukai transform from \cite[Thm.~2.9]{BES20} to construct a \emph{free monad} for $M$\!, which will be a virtual resolution of this form when $M$ is $\zero$-regular. In Section~\ref{sec:reg-to-qlin} we will use our uniqueness theorem to prove that this virtual resolution is isomorphic to the minimal free resolution of the truncation $M_{\geq\dd}$.

While a minimal free resolution $F_\b$ of an $S$-module $M$ can be easily computed using Gr\"obner methods whenever $S$ is the Cox ring of a smooth projective toric variety $X$\!, the complex $F_\b$ does not provide a faithful reflection of the geometry of $\shM$ even in the case $X = \PP\nn$\!. For example, when the Picard rank of $X$ is greater than one, the length of $F_\b$ generally exceeds $\dim X$\!.\linebreak To bridge this gap, Berkesch, Erman, and Smith introduced \emph{virtual resolutions} in \cite{BES20}.

\begin{definition}
  A $\Pic(X)$-graded complex of free $S$-modules $G_\b$ is a \emph{virtual resolution} of the module $M$ if the complex $\tilde G_\b$ of locally free sheaves on $X$ is a resolution of the sheaf $\shM$\!.
\end{definition}

Virtual resolutions have advantages over ordinary free resolutions, for instance in achieving better analogues of Hilbert's syzygy theorem (e.g.~\cites{Yang21,HNVT22}), reflecting the vanishing of sheaf $\Tor$ and $\Ext$ \cite[Cor. 5.2 and~5.8]{BKLY21}, and distinguishing sheaves represented by the same module over toric varieties with the same Cox ring. However, virtual resolutions are often less rigid: there are many non-isomorphic virtual resolutions for each module $M$\!.

We are particularly interested in virtual resolutions arising from a Fourier--Mukai functor whose output is a complex concentrated from homological index $-\!\dim X$ to $\dim X$\!. Following \cite[\S8]{EES15}, we define a \emph{free monad} of a coherent sheaf $\cF$ to be a finite complex
\[ \cL\colon 0\gets\cL_{-s}\gets\cdots\gets\cL_{-1}\gets\cL_0\gets\cL_1\gets\cdots\cL_t\gets0 \]
whose terms are direct sums of line bundles and whose homology is $\HH_\b(\cL) = \HH_0(\cL) \simeq \cF$\!. Though there are many references on Fourier--Mukai transforms as functors between derived categories, for our purposes we need a rigorous construction of a specific representative of the output complex up to homotopy. Thus we include in Appendix~\ref{sec:appendix} the details of the Beilinson spectral sequence behind the construction below.

% The main results of this section are Proposition~\ref{prop:free-monad} and Theorem~\ref{thm:uniqueness}, which describe the Betti numbers of a free monad for $\shM$ constructed from the Beilinson spectral sequence (cf.\ \cite[Thm.~2.9]{BES20}) and prove its uniqueness. When $M$ is $\zero$-regular this free monad is a virtual resolution, which we will use in Section~\ref{sec:regularity-criterion}.

\begin{proposition}\label{prop:free-monad}
  Let $M$ be a finitely generated $\ZZ^r$-graded $S$-module.  There is a free monad $\cL$ for $\shM$ with terms
  \( \displaystyle \cL_k = \bigoplus_{|\aa|\geq k} \OO_{\PP\nn}(-\aa)\otimes\HH^{|\aa|-k}(\PP\nn\!, \shM\otimes\Om_{\PP\nn}^\aa(\aa)) \)
  so that
  \begin{enumerate}
  \item the free complex $G_\b = \Gs(\cL)$ has Betti numbers $\beta_{k,\aa}(G_\b) = h^{|\aa|-k}\big(\PP\nn\!, \shM\otimes\Om_{\PP\nn}^\aa(\aa)\big)$;
  \item\label{item:virtual-resolution} if $\HH^i\big(\PP\nn\!, \shM\otimes\Om_{\PP\nn}^{\aa}(\aa)\big) = 0$ for $i>|\aa|$ then $G_\b$ is a virtual resolution for $M$ whose differentials have no unit maps.
  \end{enumerate}
\end{proposition}
\begin{proof}
  Let $\cK$ be the resolution of the diagonal from \eqref{eq:res-o'diag-terms} and let $\Phi_\cK$ be the corresponding Fourier--Mukai transform. The first page $E_1^{-s,t}$ of the Beilinson spectral sequence after taking homology vertically is:
  \begin{equation}\label{eq:derived-E1}
    \begin{tikzcd}[column sep=small, row sep=tiny]
	{} & \vdots & \vdots & \vdots & {} \\[-5pt]
	{} & {\R^2p_*(q^*\shM\otimes\cK_0)} & {\R^2p_*(q^*\shM\otimes\cK_1)} & {\R^2p_*(q^*\shM\otimes\cK_2)} & \cdots \\
	{} & {\R^1p_*(q^*\shM\otimes\cK_0)} & {\R^1p_*(q^*\shM\otimes\cK_1)} & \dao{\R^1p_*(q^*\shM\otimes\cK_2)} & \cdots \\
	{} & \PG{ p_*(q^*\shM\otimes\cK_0)} & \PG{ p_*(q^*\shM\otimes\cK_1)} & \PG{ p_*(q^*\shM\otimes\cK_2)} & \cdots \\
	{} & {} & {}
	\arrow[from=4-4, to=4-3]
	\arrow[from=4-3, to=4-2]
	\arrow[from=4-5, to=4-4]
	\arrow[from=3-4, to=3-3]
	\arrow[from=3-5, to=3-4]
	\arrow[from=2-5, to=2-4]
	\arrow[from=2-4, to=2-3]
	\arrow[from=2-3, to=2-2]
	\arrow[from=3-3, to=3-2]
  \arrow[crossing over, shift left=16, shorten <=-28pt, shorten >=-18pt, dashed, no head, from=4-2, to=1-2]
	\arrow[shift right=5, shorten <=-87pt, shorten >=-20pt, dashed, no head, from=4-2, to=4-5]
	\arrow["{k=2}"{description, sloped, pos=0.70}, shift left=4, shorten >=70pt, dotted, no head, from=3-5, to=5-3]
	\arrow["{k=1}"{description, sloped, pos=0.81}, shift left=4, shorten >=70pt, dotted, no head, from=2-5, to=5-2]
	\arrow["{k=0}"{description, sloped, pos=0.95}, shift left=1, shorten >=30pt, dotted, no head, from=1-5, to=5-1]
	\arrow[shift right=2, shorten <=25pt, shorten >=23pt, dotted, no head, from=1-3, to=3-1]
	\arrow[shift right=2, shorten <=25pt, shorten >=23pt, dotted, no head, from=1-4, to=4-1]
    \end{tikzcd}
  \end{equation}

  The vertical complexes of $E_0$ in \eqref{eq:beilinson-E0} are sheaves tensored with complexes of vector spaces that are global sections of \v{C}ech complexes, so they satisfy the splitting hypotheses of \cite[Lem.~3.5]{EFS03}, which implies that the total complex of $E_0$ is homotopy equivalent to a complex $\cL$ with terms \( \cL_k = \bigoplus_{s-t = k} E_1^{-s,t} \). Hence
  \[ \cL\sim\Tot(E_0) = \Phi_\cK(\shM)\sim\shM. \]
  Since the terms of $E_1$ are direct sums of line bundles, the complex $\cL$ is a free monad for $\shM$\!.

  Observe that the only terms with twist $\aa$ appear in $\cK_s$ for $s=|\aa|$ and that the Betti numbers in homological index $k$ come from the higher direct images $E_1^{-s,t}$ on diagonals with $s-t=k$. Hence \( \beta_{k,\aa}(G_\b) \) is the rank of $\OO_{\PP\nn}(-\aa)$ in \( E_1^{-|\aa|,|\aa|-k} \) which is \( h^{|\aa|-k}(\PP\nn\!, \shM\otimes\Om^\aa(\aa)) \).

  Lastly, note that the hypothesis of part \eqref{item:virtual-resolution} implies that the terms of \eqref{eq:derived-E1} on diagonals with $k<0$ vanish; hence the free monad $\cL$ is a locally free resolution. Since each map in the construction from \cite[Lem.~3.5]{EFS03} increases the value of $-s$, the differentials in $G_\b$ have no degree $\zero$ components.
\end{proof}

\begin{remark}\label{rem:BES-1.2}
  In the proof of \cite[Prop.~1.2]{BES20}, Berkesch, Erman, and Smith show that if $M$ is sufficiently twisted so that all higher direct images of $\shM\otimes\Om_{\PP\nn}^\aa(\aa)$ vanish, then the $E_1$ page will be concentrated in one row, which results in a linear virtual resolution. Similarly in \cite[Prop.~1.7]{EES15}, Eisenbud, Erman, and Schreyer prove that for sufficiently positive twists, the truncation of $M$ has a linear free resolution. However, in both cases the positivity condition is stronger than $\zero$-regularity for $M$\!, as illustrated by the following example.
\end{remark}

\begin{example}
	Write $S = \kk[x_0,x_1,y_0,y_1,y_2]$ for the Cox ring of $\PP1\times\PP2$ and consider the ideal $I = \langle y_0 + y_1 + y_2, x_1 y_2 \rangle$. Then $M = S/I$ is a bigraded, $(0,0)$-regular $S$-module. The global sections of the Beilinson spectral sequence for $\shM$ has first page
%https://q.uiver.app/?q=WzAsMTAsWzIsMSwiUygwLC0xKSJdLFszLDAsIlMoLTEsLTEpIl0sWzUsMCwiUygtMSwtMikiXSxbMywxLCIwIl0sWzIsMCwiMCJdLFswLDAsIjAiXSxbNSwxLCIwIl0sWzYsMCwiICAwIl0sWzYsMSwiMCJdLFswLDEsIlMiXSxbMiwxLCJbeF8yK3hfMyt4XzRdIiwyXSxbMywwXSxbMSw0XSxbNiwzXSxbNywyXSxbOCw2XSxbMiwwLCIteF8xeF80IiwxLHsic3R5bGUiOnsiYm9keSI6eyJuYW1lIjoiZG90dGVkIn19fV0sWzQsNV0sWzAsOSwiW3hfMit4XzMreF80XSJdLFsxLDksInhfMXhfNCIsMSx7InN0eWxlIjp7ImJvZHkiOnsibmFtZSI6ImRvdHRlZCJ9fX1dLFs5LDUsIiIsMCx7Im9mZnNldCI6LTQsInNob3J0ZW4iOnsic291cmNlIjozMH0sInN0eWxlIjp7ImJvZHkiOnsibmFtZSI6ImRhc2hlZCJ9LCJoZWFkIjp7Im5hbWUiOiJub25lIn19fV0sWzksOCwiIiwwLHsib2Zmc2V0Ijo1LCJzdHlsZSI6eyJib2R5Ijp7Im5hbWUiOiJkYXNoZWQifSwiaGVhZCI6eyJuYW1lIjoibm9uZSJ9fX1dXQ==
  \[\begin{tikzcd}[row sep=small]
	0 && 0 & \dao{S(-1,-1)} && \dao{S(-1,-2)} & {0} \\
	\PG{S} && \PG{S(0,-1)} & 0 && 0 & 0
	\arrow["{y_0+y_1+y_2}"', from=1-6, to=1-4]
	\arrow[from=2-4, to=2-3]
	\arrow[from=1-4, to=1-3]
	\arrow[from=2-6, to=2-4]
	\arrow[from=1-7, to=1-6]
	\arrow[from=2-7, to=2-6]
	\arrow["{-x_1y_2}"{description}, dotted, from=1-6, to=2-3]
	\arrow[from=1-3, to=1-1]
	\arrow["{y_0+y_1+y_2}", from=2-3, to=2-1]
	\arrow["{x_1y_2}"{description}, dotted, from=1-4, to=2-1]
	\arrow[shift left=4, shorten =-24pt, dashed, no head, from=2-1, to=1-1]
	\arrow[shift right=5, shorten =-21pt, dashed, no head, from=2-1, to=2-7]
  \end{tikzcd}\]
  where the dotted diagonal maps are lifts of maps from the second page of the spectral sequence, which agree with the maps from \cite[Lem.~3.5]{EFS03}.
\end{example}

In the next section we state and prove Theorem~\ref{main-thm-reg} by making explicit the restrictions on the virtual resolution above that follow from the regularity of $\shM$ and using them to bound the shape of the minimal free resolution of a truncation of $M$\!. In a sense, we will characterize $\dd$-regularity by showing that this virtual resolution is isomorphic to the minimal free resolution of $M_{\geq\dd}$. To do so we will need the following uniqueness theorem for virtual resolutions which consist only of twists in Beilinson's exceptional collection (see Example~\ref{ex:hyperelliptic-curve}).

\begin{theorem}\label{thm:uniqueness}
  Let $F_\b$ and $G_\b$ be virtual resolutions of an $S$-module. If $F_\b$ and $G_\b$ have no unit maps and their terms are direct sums of $S(-\aa)$ for $\zero\leq\aa\leq\nn$, then they are isomorphic.
  %% For arXiv sleuths:
  %%  Let $X$ be a smooth projective toric variety and $E_{1},\ldots,E_{t}$ a full strong exceptional collection of line bundles on $X$\!. Suppose $F_{\b}$ and $G_{\b}$ are $\Pic(X)$-graded minimal complexes of $S$-modules. If $\tilde F_\b$ and $\tilde G_\b$ are quasi-isomorphic as complexes of $\OO_{X}$-modules and every term of $\tilde F_\b$ and $\tilde G_\b$ is a direct sum of $E_{i}$ for $0\leq i \leq t$, then $G_\b$ and $F_\b$ are isomorphic.
\end{theorem}

The proof uses the tilting bundle $\cE \coloneq \bigoplus_{\aa=\zero}^\nn \OO(-\aa)$ from Beilinson's exceptional collection and the induced equivalence of categories described in Section~\ref{sec:beilinson-quiver} to reduce the question to the uniqueness of minimal projective resolutions over the endomorphism algebra $A = \End(\cE)$. Another ingredient of the proof is that the functor $\Hom(\cE,-)$ is exact on the class of free monads constructed in Proposition~\ref{prop:free-monad}, and moreover that we may apply the functor $\RHom(\cE,-)$ from~\eqref{eq:RHom} term-wise on locally free resolutions in this class to yield projective resolutions. See \cite[\S2.5]{Weibel1994} for terminology on derived categories and acyclic classes.

\begin{lemma}\label{lem:Hom-acyclic}
 Let $\cL$ be a complex whose terms consist of summands of $\cE$. Then
 \begin{enumerate}
 \item\label{item:free-monad} if $\cL$ is a free monad then $\Hom(\cE,\cL) = \RHom(\cE,\cL) \sim \RHom(\cE, \HH_0(\cL))$;
 \item\label{item:resolution} if $\cL$ is a resolution then $\Hom(\cE,\cL)$ is a projective resolution of an $A$-module.
 \end{enumerate}
\end{lemma}
\begin{proof}
  Since summands of $\cE$ form a strong exceptional sequence and each $\cL_j$ consists of summands of $\cE$ we have $\R^i\Hom(\cE,\cL_j) = 0$ for $i\neq0$. Hence $\cL$ is $F$-acyclic for $F=\Hom(\cE,-)$ and the hypercohomology spectral sequence
  \[ E^{-j,i} = \R^i\Hom(\cE,\cL_j) \Rightarrow \RHom(\cE,\cL) \]
  degenerates on the first page to $\Hom(\cE,\cL)$. Moreover, since $\RHom$ preserves quasi-isomorphisms
  \[ \Hom(\cE,\cL) = \RHom(\cE,\cL) \sim \RHom(\cE,\HH_0(\cL)). \]
  Thus we can apply the functor $\RHom(\cE,-)$ from~\eqref{eq:RHom} by applying $\Hom(\cE,-)$ on such free monads term-wise.

  For the second part, recall that each $\RHom(\cE,\cL_j)$ is a projective $A$-module and suppose the free monad $\cL$ has homology $\HH_\b(\cL) = \HH_0(\cL) = \cG$. Thus the projective complex $\Hom(\cE,\cL)$ is quasi-isomorphic to an $A$-module if
  \[ \HH_i(\Hom(\cE,\cL)) = \R^i\Hom(\cE,\cG) = 0 \text{ for } i\neq0.\]
  To see that this vanishing holds when $\cL$ is a locally free resolution we induct on its length: the long exact sequence for $\Hom(\cE,-)$ applied to the length 1 resolution
  \[ 0 \gets \cG \gets \cL_0 \gets \cL_1 \gets 0 \]
  implies that $\R^i\Hom(\cE,\cG) = \R^{i+1}\Hom(\cE,\cL_1) = 0$ for $i\neq0$. If $\cL$ has length $\ell$, break it into a short exact sequence and a length $\ell-1$ resolution
 \[ 0 \gets \cG \gets \cL_0 \gets \cG' \gets 0 \quad \text{and} \quad 0 \gets \cG' \gets \cL_1 \gets \cdots \gets \cL_\ell \gets 0. \]
 Then by the inductive hypothesis $\R^i\Hom(\cE,\cG')=0$ for $i\neq0$, so $\R^i\Hom(\cE,\cG)=0$ for $i\neq0$ using the short exact sequence. Thus when $\cL$ is a locally free resolution as above, $\Hom(\cE,\cL) \sim \RHom(\cE,\cG)$ is a projective resolution of an $A$-module.
\end{proof}

%% \begin{lemma}\label{lem:Hom-acyclic}
%%   If $\cL$ is a locally free resolution of $\cG$ whose terms consist of summands of $\cE$ then $\RHom(\cE,\cL)=\Hom(\cE,\cL)$ is a projective resolution over $A$.
%% \end{lemma}
%% \begin{proof}
%%   Since $\RHom$ preserves quasi-isomorphisms we have $\RHom(\cE,\cL)\sim\RHom(\cE,\cG)$, so $\RHom(\cE,\cL)$ is a resolution if and only if $\R^i\Hom(\cE,\cG)=0$ for all $i>0$.  We will show that this vanishing holds by inducting on the length of $\cL$: consider \( 0 \gets \cG \gets \cL_0 \gets \cL_1 \gets 0 \) with length 1. Then the long exact sequence of $\Hom(\cE,-)$ implies that $\R^i\Hom(\cE,\cG) = \R^{i+1}\Hom(\cE,\cL_1)=0$ for $i\neq0$. If $\cL$ has length $\ell$ break it into %an exact sequence and a resolution of length $\ell-1$
%%   \[ 0 \gets \cG \gets \cL_0 \gets \cG' \gets 0 \quad \text{and} \quad 0 \gets \cG' \gets \cL_1 \gets \cdots \gets \cL_\ell \gets 0. \]
%%   Then by the inductive hypothesis $\R^i\Hom(\cE,\cG')=0$ for $i\neq0$, so $\R^i\Hom(\cE,\cG)=0$ for $i\neq0$ using the short exact sequence. Thus when $\cL$ is a locally free resolution as above, $\RHom(\cE,\cG)=\Hom(\cE,\cL)$ is a projective resolution over $A$.
%% \end{proof}

\begin{proof}[Proof of Theorem~\ref{thm:uniqueness}]
  Consider the (non-commutative) diagram
  \[\begin{tikzcd}[column sep = large]
    {\mathrm{Ch}^\b(\Coh(X))} & {\mathrm{Ch}^\b(\mathrm{mod-}A)} \\
	  {                \Db{X}}  &            {\Db{\mathrm{mod-}A}}.
	  \arrow[from=1-1, to=2-1]
	  \arrow["{\RHom(\cE,-)}", from=2-1, to=2-2]
	  \arrow["{\Hom(\cE,-)}", from=1-1, to=1-2]
	  \arrow[from=1-2, to=2-2]
  \end{tikzcd}\]
  Using Lemma~\ref{lem:Hom-acyclic}\eqref{item:free-monad} this diagram commutes for free monads in $\mathrm{Ch}^\b(\Coh(X))$; hence we may compute the functor $\RHom(\cE,-)$ from~\eqref{eq:RHom} by applying $\Hom(\cE,-)$ to $\tilde F_\b$ and $\tilde G_\b$. This yields quasi-isomorphic minimal complexes of $A$-modules $C_\b$ and $D_\b$, respectively, which by Lemma~\ref{lem:Hom-acyclic}\eqref{item:resolution} are projective.  Hence derived maps in $\Db{\mathrm{mod-}A}$ between $C_\b$ and $D_\b$ lift to maps of complexes. % in $\mathrm{Ch}^\b(\mathrm{mod-}A)$.

  Thus there are quasi-isomorphisms $f\colon C_\b\to D_\b$ and $g\colon D_\b\to C_\b$ with chain homotopies
  \[ fg - \mathrm{id}_D = d_Dh + hd_D \quad \text{and} \quad gf - \mathrm{id}_C = d_Ch' + h'd_C.\]
  Since the differential $d_D$ is minimal, $fg - \mathrm{id}_D = d_Dh + hd_D$ is minimal, so an application of Nakayama's lemma \cite[Lem.~3.2.4]{DW2017} shows that $fg$ is surjective. By the analogous argument on the other chain homotopy $gf$ is also surjective. Since a surjection of finite-dimensional vector spaces is injective, both $fg$ and $gf$, and hence both $f$ and $g$, are isomorphisms.
  % Explicitly: let $J$ be the Jacobson radical of $A$, then Nakayama's lemma implies that since the induced map $\bar{fg} = \bar{dh + hd - \mathrm{id}_D} = \bar{\mathrm{id}_D} = \id_{D/J}$ is surjective, the map $fg$ is surjective, but a surjection of finite dimensional vector spaces is injective, so $fg$ is an isomorphism.

  Since $C_\b$ and $D_\b$ are projective, applying the reverse equivalence $-\Lotimes\cE$ to $f$ gives an isomorphism of complexes of $\OO_X$-modules $\tilde F_\b\to\tilde G_\b$ as desired. Finally, since $F_\b$ and $G_\b$ are free complexes, applying the twisted global sections functor $\Gs$ yields an isomorphism $F_\b \to G_\b$.
\end{proof}

\begin{remark}
  Note that Lemma~\ref{lem:Hom-acyclic} holds in more generality for any triangulated category and any additive left exact functor $F$ such that $\R^iF(\cE)=0$ for $i>0$, and the equivalence of categories applies to any full strong exceptional collection of line bundles on a toric variety. In particular, while general free monads are not sent to projective monads under the functor~\eqref{eq:RHom}, a similar uniqueness theorem holds for free monads consisting of the summands of $\cE\otimes\OO(-\dd)$ for any $\dd\in\Pic X$\!. We will show that if $\dd\in\reg(M)$ then both the Fourier--Mukai transform of $\shM$ and the minimal free resolution of $M_{\geq\dd}$ satisfy this condition.
\end{remark}

%%%%%%%%%%%%%%%%%%%%%%%%%%%%%%%%%%%%%%%%%%%%%%%%%%%%%%%%%%%%%%%%%%%%%%%%%%%%%%%%%%%%%

\section{A Criterion for Multigraded Regularity}\label{sec:regularity-criterion}

To investigate the relationship between multigraded regularity and resolutions of truncations we first need to establish a definition of linearity for a multigraded resolution.  We would like the differentials to be given by matrices with entries of total degree at most 1.  However, we will examine only the twists in the resolution, requiring that they lie in the $L$ regions from Section~\ref{sec:regularity}.  In particular, we will identify a complex with adjacent terms differing in degree by more than one as nonlinear even if the map between these terms is zero.

\begin{definition}
	Let $F_\b$ be a $\ZZ^r$-graded free resolution.  We say $F_\b$ is \emph{linear} if $F_0$ is generated in a single multidegree $\dd$ and the twists appearing in $F_j$ lie in $L_j(-\dd)$.
\end{definition}

We require $F_0$ to be generated in a single degree so that the truncation of a module with a linear resolution also has a linear resolution (see Proposition~\ref{prop:invariance}). Otherwise, for instance, the minimal free resolution of $M$ in the following example would be considered linear, yet the free resolution of its truncation $M_{\geq(1,0)}$ would not.

\begin{example}\label{ex:not-linear}
	Write $S = \kk[x_0,x_1,y_0,y_1]$ for the Cox ring of $\PP1\times\PP1$ and let $M$ be the module with resolution $S(0,-1)^2\oplus S(-1,0)^2\gets S(-1,-1)^4\gets 0$ given by the presentation matrix
	\[\begin{bmatrix}
		 x_0 &  x_1 &  0   &  0   \\
		 0   &  0   &  x_1 &  x_0 \\
		-y_0 &  0   & -y_0 &  0   \\
		 0   & -y_1 &  0   & -y_1
	\end{bmatrix}.\]
	A \textit{Macaulay2} computation shows that $M$ is $(1,0)$-regular. However, the minimal graded free resolution of the truncation $M_{\geq(1,0)}$ is
	\[\begin{tikzcd}
	0&\lar S(-1,0)^2&\lar S(-2,-1)^2&\lar 0
	\end{tikzcd}\]
	which is not linear because $(-2,-1)\notin L_1(-1,0)$.
\end{example}

This example shows that a module can be $\dd$-regular yet have a nonlinear resolution for $M_{\geq\dd}$.  Thus in order to characterize regularity in terms of truncations we need to weaken the definition of linearity.  We will use the larger $Q$ regions from Section~\ref{sec:regularity} in order to allow some maps of higher degree.

\begin{definition}\label{def:quasilinear-resolution}
	Let $F_\b$ be a $\ZZ^r$-graded free resolution. We say $F_\b$ is \emph{quasilinear} if $F_0$ is generated in a single multidegree $\dd$ and for each $j$ the twists appearing in $F_j$ lie in $Q_j(-\dd)$.
\end{definition}

\begin{example}
  Unlike on a single projective space, the resolution of $S/B$ for the irrelevant ideal $B$ on a product of projective spaces is not linear.  However it is quasilinear.  On $\PP1\times\PP2$\!, for instance, $S/B$ has resolution
	\[\begin{tikzcd}[column sep=2em]
	0 & \lar S & \lar S(-1,-1)^6 & \lar
	  \begin{matrix}
	    S(-1,-2)^6\\[-3pt]
	    \oplus \\[-3pt]
	    S(-2,-1)^3
	  \end{matrix}
	  &
	  \lar
	  \begin{matrix}
	    S(-1,-3)^2\\[-3pt]
	    \oplus \\[-3pt]
	    S(-2,-2)^3
	  \end{matrix}
	  &\lar
	  S(-2,-3)
	  & \lar 0,
	\end{tikzcd}\]
  which has generators in degree $(0,0)$ and relations in degree $(1,1)$.  Thus the resolution is not linear, since $(-1,-1)\notin L_1(0,0)$.  However $(-1,-1)\in Q_1(0,0)$ is compatible with quasilinearity.
\end{example}

This condition is inspired by \cite[Thm.~2.9]{BES20}, which characterized regularity in terms of the existence of virtual resolutions with Betti numbers similar to those of $S/B$---see Section~\ref{sec:qlin-to-reg} for a more complete discussion. Note that both linearity and quasilinearity reduce to the standard definition of linearity on a single projective space. As one might expect from that setting, they satisfy the property below, which will be proven as a corollary of Theorems \ref{thm:betti-truncations} and \ref{thm:betti-regularity} in Section~\ref{sec:betti-inner-bounds}.

\begin{proposition}\label{prop:invariance}
	Let $M$ be a $\ZZ^r$-graded $S$-module.  If $M_{\geq\dd}$ has a linear (respectively quasilinear) resolution and $\dd'\geq\dd$ then $M_{\geq\dd'}$ has a linear (respectively quasilinear) resolution.
\end{proposition}

%\mahrud{Research: can a non-saturated module have a linear resolution but not be regular?}

A linear resolution for $M_{\geq\dd}$ implies that $M$ is $\dd$-regular when $\HH_B^0(M)=0$.  To obtain a converse that generalizes Eisenbud--Goto's result one should instead check that the resolution is quasilinear.  This gives a criterion for regularity that does not require computing cohomology.

\begin{theorem}\label{thm:quasilinear-conjecture}
	Let $M$ be a finitely generated $\ZZ^r$-graded $S$-module such that $\HH_B^0(M)=0$.  Then $M$ is $\dd$-regular if and only if the minimal free resolution $F_\b$ of $M_{\geq\dd}$ is quasilinear and $F_0$ is generated in degree $\dd$.
\end{theorem}

We prove one direction of Theorem~\ref{thm:quasilinear-conjecture} in Section~\ref{sec:reg-to-qlin} (Theorem~\ref{thm:reg-to-qlin}) and the other in Section~\ref{sec:qlin-to-reg} (Theorem~\ref{thm:qlin-to-reg}). The two directions are illustrated in the following example, which verifies regularity or non-regularity of a module by computing a single minimal free resolution for the truncation at a given degree and checking if it is quasilinear.

\begin{example}\label{ex:hyperelliptic-curve}
	A smooth hyperelliptic curve of genus 4 can be embedded into $\PP1\times\PP2$ as a curve of degree $(2,8)$. An example of such a curve is given explicitly in \cite[Ex.~1.4]{BES20} as the $B$-saturation $I$ of the ideal
	\[ \left\langle x_0^2y_0^2+x_1^2y_1^2+x_0x_1y_2^2, x_0^3y_2+x_1^3(y_0+y_1) \right\rangle. \]
  Using Theorem~\ref{thm:reg-to-qlin} it is relatively easy to check that $S/I$ is not $(2,1)$-regular: the minimal free resolution of $(S/I)_{\geq (2,1)}$ is
	\[\begin{tikzcd}
		0 & \lar S(-2,-1)^9 & \lar
		\begin{matrix}
			S(-2,-2)^{10} \\[-3pt]
			\oplus \\[-3pt]
			\dao{S(-2,-3)^2} \\[-3pt]
			\oplus \\[-3pt]
			S(-3,-1)^7
		\end{matrix} & \lar
		\begin{matrix}
			S(-2,-3)^3 \\[-3pt]
			\oplus \\[-3pt]
			S(-3,-2)^6 \\[-3pt]
			\oplus \\[-3pt]
			S(-3,-3)^3
		\end{matrix}& \lar
		S(-3,-3)^2 & \lar 0
	\end{tikzcd}\]
	which is not quasilinear because $(-2,-3)\notin Q_1(-2,-1)$.
  %\mahrud{Research: can this be used to prove a specific cohomology non-vanishing?}
  Meanwhile, $S/I$ is $(2,2)$-regular by Theorem~\ref{thm:qlin-to-reg} because the minimal free resolution of $(S/I)_{\geq(2,2)}$ is quasilinear:
  \[\begin{tikzcd}
	0 & \lar S(-2,-2)^{17} & \lar
	\begin{matrix}
		S(-2,-3)^{26} \\[-3pt]
		\oplus \\[-3pt]
		S(-3,-2)^{15}
	\end{matrix} & \lar
	\begin{matrix}
		S(-2,-4)^9 \\[-3pt]
		\oplus \\[-3pt]
		S(-3,-3)^{22}
	\end{matrix}& \lar
	S(-3,-4)^7 & \lar 0.
	\end{tikzcd}\]
  Note that $(2,2)$-regularity also follows from \cite[Thm.~2.9]{BES20} and \cite[Ex.~2.3]{BES20}, which constructs a quasilinear virtual resolution for $S/I$ using a Fourier--Mukai transform. \linebreak In fact, since all twists are line bundles on $\PP1\times\PP2$ in the range $(-2,-2)$ to $(-2,-2)+(-1,-2)$, by Theorem~\ref{thm:uniqueness} the two complexes are isomorphic.
\end{example}

%%%%%%%%%%%%%%%%%%%%%%%%%%%%%%%%%%%%%%%%%%%%%%%%%%%%%%%%%%%%%%%%%%%%%%%%%%%%%%%%%%%%%

\subsection{Regularity Implies Quasilinearity}\label{sec:reg-to-qlin}

In Proposition~\ref{prop:free-monad} we constructed a virtual resolution with Betti numbers determined by the sheaf cohomology of $\shM\otimes\Om_{\PP\nn}^\aa(\aa)$. By resolving the sheaves $\Om_{\PP\nn}^\aa(\aa)$ in terms of line bundles and tensoring with $\shM$\!, we can relate the cohomological vanishing in the definition of multigraded regularity to the shape of this virtual resolution. The following lemma implies that when $M$ is $\dd$-regular the virtual resolution is quasilinear, i.e., the coefficients of twists outside of $Q_i(-\dd)$ are zero. The lemma is a variant of \cite[Lem.~2.13]{BES20} (see Section~\ref{sec:qlin-to-reg}).

\begin{lemma}\label{lem:BES-reg-to-qlin}
	If a $\ZZ^r$-graded $S$-module $M$ is $\zero$-regular then $\HH^{|\aa|-i}(\PP\nn\!, \shM\otimes\Om_{\PP\nn}^\aa(\aa))=0$ for all $-\aa\notin Q_i(\zero)$ and all $i>0$.
\end{lemma}
\begin{proof}
	Fix $i$ and $\aa\in\ZZ^r$ with $-\aa\notin Q_i(\zero)$, and suppose that $\HH^{|\aa|-i}(\PP\nn\!, \shM\otimes\Om_{\PP\nn}^\aa(\aa))\neq 0$.  We will show that $M$ is not $\zero$-regular.  We must have $\zero\leq\aa\leq\nn$, else $\Om_{\PP\nn}^\aa(\aa)=0$.  Let $\ell$ be the number of nonzero coordinates in $\aa$.

	A tensor product of locally free resolutions for the factors $\pi_i^*(\Om_{\PP{n_i}}^{a_i})$ gives a locally free resolution for $\Om_{\PP\nn}^\aa(\aa)$.  Since $\Om_{\PP{n_i}}^0=\OO_{\PP{n_i}}$ we can use $r-\ell$ copies of $\OO_{\PP\nn}$ and $\ell$ linear resolutions, each generated in total degree 1, to obtain such a resolution $\cF_\b$ (see Section~\ref{sec:Koszul-cotangent}).  Thus the twists in $\cF_j$ have nonpositive coordinates and total degree $-j-\ell$, so they are in $L_{j+\ell}(\zero)$.

	Since $\cF$ is locally free the cokernel of $\shM\otimes\cF$ is isomorphic to $\shM\otimes\Om_{\PP\nn}^\aa(\aa)$.  By a standard spectral sequence argument, explained in the proof of Theorem~\ref{thm:qlin-to-reg}, the nonvanishing of $\HH^{|\aa|-i}(\PP\nn\!, \shM\otimes\Om_{\PP\nn}^\aa(\aa))$ implies the existence of some $j$ such that $\HH^{|\aa|-i+j}(\PP\nn\!, \shM\otimes\cF_j)\neq 0$.

	If $i=0$ then
	\[ |\aa|-i+j\geq\ell-i+j=j+\ell. \]
	If $i>0$ then $\aa-\one$ has $\ell$ nonnegative coordinates that sum to $|\aa|-\ell$.  Thus $|\aa|-\ell>i-1$, since $-\aa\notin Q_i(\zero)=L_{i-1}(-\one)$ (see Remark~\ref{rem:describeL}).  This also gives
	\[ |\aa|-i+j \geq (\ell+i)-i+j = j+\ell, \]
	so in either case $L_{j+\ell}(\zero)\subseteq L_{|\aa|-i+j}(\zero)$.  Therefore $\HH^{|\aa|-i+j}(\PP\nn\!, \shM\otimes\cF_j)\neq 0$ for $F_j$ with twists in $L_{j+\ell}(\zero)$ implies that $M$ is not $\zero$-regular.
\end{proof}

See \cite[Thm.~5.5]{CMR07} for a similar result relating Hoffman and Wang's definition of regularity \cite{HW04} to a different cohomology vanishing for $\shM\otimes\Om_{\PP\nn}^\aa(\aa)$.  A stronger version of their definition, requiring vanishing in multiple local cohomology modules, is also equivalent to a statement about Betti numbers \cite[Thm.~4.10]{HW04}.

Motivated by the quasilinearity of the virtual resolution in Proposition~\ref{prop:free-monad}, we will prove that the $\dd$-regularity of $M$ implies that the minimal free resolution of $M_{\geq\dd}$ is quasilinear. Let $K$ be the Koszul complex from Section~\ref{sec:Koszul-cotangent} and $\cech^p(B,-)$ the \v{C}ech complex as in Section~\ref{sec:truncations-local-coh}. We will use the spectral sequence of a double complex with rows from subcomplexes of $K$ and columns given by \v{C}ech complexes in order to relate the Betti numbers of $M_{\geq\dd}$ to the sheaf cohomology of $\shM\otimes\Om_{\PP\nn}^\aa(\aa)$.

\begin{theorem}\label{thm:reg-to-qlin}
	Let $M$ be a finitely generated $\ZZ^r$-graded $S$-module such that $\HH_B^0(M)_\dd=0$.  If $M$ is $\dd$-regular then $M_{\geq\dd}$ has a quasilinear free resolution $F_\b$ with $F_0$ generated in degree $\dd$.
\end{theorem}
\begin{proof}
	Without loss of generality we may assume that $\dd=\zero$ and $M=M_{\geq\zero}$ (see Lemma~\ref{lem:truncation-regularity}).

  Let $F_\b$ be the minimal free resolution of $M$\!. We will show that the degree $\aa$ Betti numbers of $F_\b$ in homological index $j$ are given by $h^{|\aa|-j}(\PP\nn\!, \shM\otimes\Om_{\PP\nn}^\aa(\aa))$ (i.e., that $F_\b$ has the same Betti numbers as the free monad $G_\b$ from Proposition~\ref{prop:free-monad}). Since $M$ is $\zero$-regular the vanishing of these cohomology groups results in a quasilinear free resolution by Lemma~\ref{lem:BES-reg-to-qlin}, and \cite[Thm.~5.4]{MS04} forces $F_0$ to be generated in degree $\zero$. (In fact this is enough to show that $F_\b$ and $G_\b$ are isomorphic, as we will do in Corollary~\ref{cor:same-res}.)

	Fix a degree $\aa\in\ZZ^r$\!.  Construct a double complex $E^{\b,\b}$ by taking the \v{C}ech complex of each term in $M\otimes K_\b^{\leq\aa}$ and including the \v{C}ech complex of $M\otimes\hat\Om_{\PP\nn}^\aa$ as an additional column.  Index $E^{\b,\b}$ so that
	\[E^{s,t}=\begin{cases}
		\cech^t\left(B, M\otimes K_{|\aa|+1-s}^{\leq\aa}\right) & \text{if }s>0,\\
		\cech^t\left(B, M\otimes\hat\Om_{\PP\nn}^\aa\right) & \text{if }s=0.
	\end{cases}\]
	We will compare the vertical and horizontal spectral sequences of $E^{\b,\b}$ in degree $\aa$.  By Lemma~\ref{lem:finite-length-homology} and the fact that $K_\b^{\leq\aa}$ is locally free, the sheafification of the 0-th row $E^{\b,0}$ is exact.  Thus by Lemma~\ref{lem:cech-horseshoe} the rows of $E^{\b,\b}$ are exact for $t\neq 0$.

  \[\begin{tikzcd}[column sep=tiny, row sep=small]
	{} & \vdots & \vdots & \vdots && \vdots \\
	& {\cech^2(B, M\otimes\hat\Om_{\PP\nn}^\aa)} & {\cech^2(B, M\otimes K^{\leq\aa}_{|\aa|})} & {\cech^2(B, M\otimes K^{\leq\aa}_{|\aa|-1})} & \cdots & {\cech^2(B, M\otimes K^{\leq\aa}_0)} \\
	& {\cech^1(B, M\otimes\hat\Om_{\PP\nn}^\aa)} & {\cech^1(B, M\otimes K^{\leq\aa}_{|\aa|})} & {\cech^1(B, M\otimes K^{\leq\aa}_{|\aa|-1})} & \cdots & {\cech^1(B, M\otimes K^{\leq\aa}_0)} \\
	& {M\otimes\hat\Om_{\PP\nn}^\aa} & {M\otimes K^{\leq\aa}_{|\aa|}} & {M\otimes K^{\leq\aa}_{|\aa|-1}} & \cdots & {M\otimes K^{\leq\aa}_0} \\
	{} &&&&& {}
	\arrow[from=4-2, to=3-2]
	\arrow[from=3-2, to=2-2]
	\arrow[from=2-2, to=1-2]
	\arrow[from=4-3, to=3-3]
	\arrow[from=3-3, to=2-3]
	\arrow[from=2-3, to=1-3]
	\arrow[from=4-4, to=3-4]
	\arrow[from=3-4, to=2-4]
	\arrow[from=2-4, to=1-4]
	\arrow[from=4-6, to=3-6]
	\arrow[from=3-6, to=2-6]
	\arrow[from=2-6, to=1-6]
	\arrow[from=4-2, to=4-3]
	\arrow[from=4-3, to=4-4]
	\arrow[from=4-4, to=4-5]
	\arrow[from=4-5, to=4-6]
	\arrow[from=3-2, to=3-3]
	\arrow[from=3-3, to=3-4]
	\arrow[from=3-4, to=3-5]
	\arrow[from=3-5, to=3-6]
	\arrow[from=2-2, to=2-3]
	\arrow[from=2-3, to=2-4]
	\arrow[from=2-4, to=2-5]
	\arrow[from=2-5, to=2-6]
	\arrow[shift  left=5, shorten <=10pt, shorten >=-48pt, dashed, no head, from=5-1, to=5-6]
	\arrow[shift right=5, shorten <=12pt, shorten >=-20pt, dashed, no head, from=5-1, to=1-1]
  \end{tikzcd}\]

  %% \[\begin{tikzcd}[column sep=tiny, row sep=small]
	%% {} & \vdots & & \vdots & \vdots & \vdots \\
	%% & {\cech^2(B, M\otimes K^{\leq\aa}_0)} & \cdots & {\cech^2(B, M\otimes K^{\leq\aa}_{|\aa|-1})} & {\cech^2(B, M\otimes K^{\leq\aa}_{|\aa|})} & {\cech^2(B, M\otimes\hat\Om_{\PP\nn}^\aa)} \\
	%% & {\cech^1(B, M\otimes K^{\leq\aa}_0)} & \cdots & {\cech^1(B, M\otimes K^{\leq\aa}_{|\aa|-1})} & {\cech^1(B, M\otimes K^{\leq\aa}_{|\aa|})} & {\cech^1(B, M\otimes\hat\Om_{\PP\nn}^\aa)} \\
	%% & {M\otimes K^{\leq\aa}_0} & \cdots & {M\otimes K^{\leq\aa}_{|\aa|-1}} & {M\otimes K^{\leq\aa}_{|\aa|}} & {M\otimes\hat\Om_{\PP\nn}^\aa} \\
	%% {} &&&&& {}
  %% % vertical arrows
	%% \arrow[from=4-2, to=3-2]
	%% \arrow[from=3-2, to=2-2]
	%% \arrow[from=2-2, to=1-2]
	%% \arrow[from=4-4, to=3-4]
	%% \arrow[from=3-4, to=2-4]
	%% \arrow[from=2-4, to=1-4]
	%% \arrow[from=4-5, to=3-5]
	%% \arrow[from=3-5, to=2-5]
	%% \arrow[from=2-5, to=1-5]
	%% \arrow[from=4-6, to=3-6]
	%% \arrow[from=3-6, to=2-6]
	%% \arrow[from=2-6, to=1-6]
  %% % horizontal arrows
	%% \arrow[from=2-3, to=2-2]
	%% \arrow[from=2-4, to=2-3]
	%% \arrow[from=2-5, to=2-4]
	%% \arrow[from=2-6, to=2-5]
	%% \arrow[from=3-3, to=3-2]
	%% \arrow[from=3-4, to=3-3]
	%% \arrow[from=3-5, to=3-4]
	%% \arrow[from=3-6, to=3-5]
	%% \arrow[from=4-3, to=4-2]
	%% \arrow[from=4-4, to=4-3]
	%% \arrow[from=4-5, to=4-4]
	%% \arrow[from=4-6, to=4-5]
	%% \arrow[shift  left=5, shorten <=10pt, shorten >=-48pt, dashed, no head, from=5-1, to=5-6]
	%% \arrow[shift right=15, shorten <=12pt, shorten >=-20pt, dashed, no head, from=5-6, to=1-6]
  %% \end{tikzcd}\]

	Since the elements of $M$ have degrees $\geq\zero$, the elements of degree $\aa$ in $M\otimes K_\b$ come from elements of degree $\leq\aa$ in $K_\b$.  Thus by Lemma~\ref{lem:finite-length-homology} the homology of $M\otimes K_\b^{\leq\aa}$ in degree $\aa$ is the same as that of $M\otimes K_\b$.  Hence the cohomology of the 0-th row $E^{\b,0}$ in degree $\aa$ computes the degree $\aa$ Betti numbers of $F_j$ for $0\leq j\leq|\aa|$, i.e., for $s>0$,
	\begin{align}\label{eq:betti-of-F}
		\HH^s(E^{\b,0})_\aa=\Tor_{|\aa|+1-s}^S(M,\kk)_\aa.
	\end{align}

	The vertical cohomology of $E^{\b,\b}$ gives the local cohomology of the terms of $M\otimes K_\b^{\leq\aa}$ along with $M\otimes\hat\Om_{\PP\nn}^\aa$.  Consider the degree $\aa$ part of this double complex.  The cohomology coming from $M\otimes K_\b^{\leq\aa}$ has summands of the form $\HH_B^i(M(-\bb))_\aa=\HH_B^i(M)_{\aa-\bb}$ where $\bb\leq\aa$.  These vanish because $M$ is $\zero$-regular, except possibly $\HH_B^0(M)_\zero$ which vanishes by hypothesis, so the only nonzero terms come from $M\otimes\hat\Om_{\PP\nn}^\aa$.

	Since $K_\b^{\leq\aa}$ is a resolution of $\kk$ in degrees $\leq\aa$, there are no elements of degree $\aa$ in $M\otimes\hat\Om_{\PP\nn}^\aa$.  Hence, using \eqref{eq:gamma-star},
	\[\HH_B^1\left(M\otimes\hat\Om_{\PP\nn}^\aa\right)_\aa = \HH^0\left(\PP\nn\!, \shM\otimes\Om_{\PP\nn}^\aa(\aa)\right).\]
	Therefore the cohomology of the $0$-th column $E^{0,\b}$ in degree $\aa$ is
	\begin{align}\label{eq:betti-of-G}
		\HH^t(E^{0,\b})_\aa = \HH_B^t(M\otimes\hat\Om_{\PP\nn}^\aa)_\aa = \HH^{t-1}(\PP\nn\!, \shM\otimes\Om_{\PP\nn}^\aa(\aa))
	\end{align}
	for $t>0$ (i.e., the Betti numbers of $G_\b$ indexed differently).

	Since both spectral sequences of the double complex $E^{\b,\b}$ converge after the first page, their total complexes agree in degree $\aa$, so by equating the dimensions of \eqref{eq:betti-of-F} and \eqref{eq:betti-of-G} in total degree $|\aa|+1-j$ we get
  \begin{align}\label{eq:magic-equality-2}
    \dim_\kk\Tor_j^S(M, \kk)_\aa = \dim_\kk\HH^{|\aa|-j}(\PP\nn\!, \shM\otimes\Om_{\PP\nn}^\aa(\aa))
  \end{align}
	for $|\aa|\geq j\geq 0$. When $j>|\aa|$, the left hand side vanishes by minimality of $F_\b$, and when $\aa$ has $\Om_{\PP\nn}^\aa=0$ the argument above still holds. Hence $F_\b$ is quasilinear by Lemma~\ref{lem:BES-reg-to-qlin} and $F_0$ is generated in degree $\zero$ by \cite[Thm.~5.4]{MS04}, as desired.
\end{proof}

The proof of Theorem~\ref{thm:reg-to-qlin} implies that when $M$ is $\dd$-regular the minimal free resolution $F_\b$ of $M_{\geq\dd}$ has the same Betti numbers as the virtual resolution $G_\b$ constructed in Proposition~\ref{prop:free-monad}. The following corollary shows that they are in fact isomorphic. In other words, the minimal free resolution of $M_{\geq\dd}(\dd)$ is composed of the terms of the Beilinson spectral sequence for $M(\dd)$, giving a concrete construction of the abstractly defined virtual resolutions used in \cite[Thm.~2.9]{BES20} to witness the regularity of $M(\dd)$ (see Example~\ref{ex:hyperelliptic-curve}).

\begin{corollary}\label{cor:same-res}
	The complexes $F_\b$ and $G_\b$ in the proof of Theorem~\ref{thm:reg-to-qlin} are isomorphic.
\end{corollary}
\begin{proof}
	From Proposition~\ref{prop:free-monad} and the fact that $\Omega_{\PP\nn}^\aa$ is nonzero only for $\zero\leq\aa\leq\nn$ it follows that $G_\b$ is a virtual resolution with no unit maps consisting of twists $S(-\aa)$ with $\zero\leq\aa\leq\nn$. Therefore the isomorphism follows from Theorem~\ref{thm:uniqueness}.
\end{proof}

To check that a module $M$ is $\dd$-regular directly from Definition~\ref{def:regular}, condition \eqref{item:reg-def} requires one to show that $H^{i}_{B}(M)_{\pp}$ vanishes for all $i>0$ and all $\pp\in\bigcup_{|\lambda|=i}\left(\dd-\lambda_1\ee_1-\cdots-\lambda_r\ee_r+\NN^r\right)$ with $\lambda\in\NN^r$\!. The proof of Theorem~\ref{thm:reg-to-qlin}, when combined with Theorem~\ref{thm:quasilinear-conjecture} and Lemma~\ref{lem:BES-reg-to-qlin}, shows that on a product of projective spaces the full strength of this condition is unnecessary. In particular, one only needs to consider $\lambda_j$ with $\lambda_j\leq n_j+1$. See \cite[Prop.~2.5]{BM11} for a similar statement on a product of two projective spaces using the definition from \cite{HW04}.

\begin{proposition}\label{prop:dan-def}
  Let $M$ be a finitely generated $\ZZ^r$-graded $S$-module.  If
  \begin{enumerate}
  \item  $\HH_B^0(M)_\pp=0$ for all $\pp\geq\dd$
  \item  $\HH_B^i(M)_\pp = 0$ for all $i>0$ and all \label{item:Daniel-def} $\pp\in\bigcup_{|\lambda|=i}\left(\dd-\sum_1^r\lambda_j\ee_j+\NN^r\right)$	where $0\leq\lambda_j\leq n_j+1$
  \end{enumerate}
  then $M$ is $\dd$-regular.
\end{proposition}
\begin{proof}
  The only difference between \eqref{item:Daniel-def} above and condition \eqref{item:reg-def} in Definition~\ref{def:regular} is the restriction to $\lambda_j\leq n_j+1$.  By the proof of Theorem~\ref{thm:reg-to-qlin}, if $\HH_B^0(M)_\bb=0$ and $M$ satisfies the hypotheses of Proposition~\ref{prop:free-monad} and Lemma~\ref{lem:BES-reg-to-qlin} then $M$ has a quasilinear resolution generated in degree $\dd$ and is thus $\dd$-regular by Theorem~\ref{thm:quasilinear-conjecture}.  In the proof of Lemma~\ref{lem:BES-reg-to-qlin} it is sufficient for the cohomology of $M(\dd)$ to vanish in degrees appearing in the resolution of some $\Om_{\PP\nn}^\aa(\aa)$, which excludes those with coordinates not $\leq\nn+\one$.
\end{proof}

\begin{example}
On $\PP{1}\times\PP{1}\times\PP{1}$\!, to show that a module $M$ is $\zero$-regular using Definition~\ref{def:regular} one must check that $H^{3}_{B}(M)_{\pp}=0$ for $\pp$ in the region with minimal elements
\[(-3,0,0),(-2,-1,0),(-2,0,-1),\ldots,(0,-3,0),\ldots,(0,0,-3).\]
However, Proposition~\ref{prop:dan-def} implies that a smaller region is sufficient. For instance, we need not check that $H^{3}_{B}(M)_{\pp}=0$ for $\pp$ equal to each of $(-3,0,0)$, $(0,-3,0)$, and $(0,0,-3)$.
\end{example}

\begin{remark}
  One may also deduce Proposition~\ref{prop:dan-def} from the proofs in \cite{BES20} without the hypothesis that $\HH_B^0(M)_\dd = 0$.
\end{remark}

%%%%%%%%%%%%%%%%%%%%%%%%%%%%%%%%%%%%%%%%%%%%%%%%%%%%%%%%%%%%%%%%%%%%%%%%%%%%%%%%%%%%%

\subsection{Quasilinearity Implies Regularity}\label{sec:qlin-to-reg}

We will now prove the reverse implication of Theorem~\ref{thm:quasilinear-conjecture}, namely that if $M_{\geq\dd}$ has a quasilinear free resolution generated in degree $\dd$ then $M$ is $\dd$-regular. We use a hypercohomology spectral sequence argument, which relates the local cohomology of $M$ to the local cohomology of the terms in a resolution for $M_{\geq\dd}$.

The following lemma will show that entire diagonals in our spectral sequence vanish when the resolution is quasilinear.  Thus the local cohomology modules $H_B^i(M)$ to which the diagonals converge also vanish in the same degrees.

\begin{lemma}\label{lem:structure-sheaf}
	If $i,j\in \NN$ then $\HH_B^{i+j+1}(S)_{\aa+\bb}=0$ for all $\aa\in L_i(\zero)$ and all $\bb\in Q_j(\zero)$.
\end{lemma}
\begin{proof}
	Note that $L_i(\zero)+Q_j(\zero)=L_i(\zero)+L_{j-1}(-\one)=L_{i+j-1}(-\one)$ as sets.  We also have $\HH_B^0(S)=\HH_B^1(S)=0$, so it suffices to show that $\HH_B^{k+1}(S)_\cc=\HH^k\left(\PP\nn\!, \OO_{\PP\nn}(\cc)\right)=0$ for $k\geq 1$ and $\cc\in L_{k-1}(-\one)$.

	The cohomology of $\OO_{\PP\nn}$ is given by the K\"{u}nneth formula.  Fix a nonempty set of indices $J\subseteq\{1,\ldots,r\}$ and consider the term
	\[\left[\bigotimes_{j\in J} \HH^{n_j}\left(\PP{n_j},\OO_{\PP{n_j}}(d_j)\right)\right]\otimes\left[\bigotimes_{j\notin J}\HH^0\left(\PP{n_j},\OO_{\PP{n_j}}(d_j)\right)\right],\]
	which contributes to $\HH^k\left(\PP\nn\!, \OO_{\PP\nn}(\cc)\right)$ for $k=\sum_{j\in J} n_j$.  It will be nonzero if and only if $d_j\leq -n_j-1$ for $j\in J$ and $d_j\geq 0$ for $j\notin J$.  If $\cc\in L_{k-1}(-\one)$ then
	\[\cc\geq-\one-\lambda_1\ee_1-\cdots-\lambda_r\ee_r\]
	for some $\lambda_i$ with $\sum\lambda_i=k-1=-1+\sum_{j\in J} n_j$.  It is not possible for the right side to have components $\leq -n_j-1$ for all $j\in J$.  Since all cohomology of $\OO_{\PP\nn}$ arises in this way, the lemma follows.
\end{proof}

In \cite[Thm.~2.9]{BES20} Berkesch, Erman, and Smith show for $M$ with $\HH_B^0(M)=\HH_B^1(M)=0$ that $M$ is $\dd$-regular if and only if $M$ has a virtual resolution $F_\b$ so that the degrees of the generators of $F(\dd)_\b$ are at most those appearing in the minimal free resolution of $S/B$.  This Betti number condition is stronger than quasilinearity, but the additional strength is not used in their proof, so the existence of such a virtual resolution is equivalent to the existence of a quasilinear one.

Since a resolution of $M_{\geq\dd}$ is a type of virtual resolution, the reverse implication of Theorem~\ref{thm:quasilinear-conjecture} mostly reduces to this result.  We present a modified proof for completeness. In particular, we do not need to require $\HH_B^1(M)=0$ because we have more information about the cokernel of our resolution.

From this perspective Theorem~\ref{thm:quasilinear-conjecture} says that the regularity of $M$ is determined not only by the Betti numbers of its virtual resolutions, but by the Betti numbers of only those virtual resolutions that are actually minimal free resolutions of truncations of $M$\!.  Thus we provide an explicit method for checking whether $M$ is $\dd$-regular.

\begin{theorem}\label{thm:qlin-to-reg}
	Let $M$ be a finitely generated $\ZZ^r$-graded $S$-module such that $\HH_B^0(M)=0$.  If $M_{\geq\dd}$ has a quasilinear free resolution $F_\b$ with $F_0$ generated in degree $\dd$, then $M$ is $\dd$-regular.
\end{theorem}
\begin{proof}
	Without loss of generality we may assume that $\dd=\zero$ and $M=M_{\geq\zero}$ (see Lemma~\ref{lem:truncation-regularity}).

	Let $F_\b$ be a quasilinear minimal free resolution of $M$\!, so that the twists of $F_j$ are in $Q_j(\zero)$.  Then the spectral sequence of the double complex $E^{\b,\b}$ with terms
	\[ E^{s,t} = \cech^t(B, F_{-s}) \]
	converges to the cohomology $\HH_B^i(M)$ of $M$ in total degree $i$.  The first page of the vertical spectral sequence has terms $\HH_B^t(F_{-s})$, so $\HH_B^{i+j}(F_j)_\aa=0$ for all $j$ (i.e., for all $(s,t)=(-j,i+j)$) implies $\HH_B^i(M)_\aa=0$.

	Therefore it suffices to show that $\HH_B^{i+j}(S(\bb))_\aa=0$ for $i\geq 1$ and all $\aa\in L_{i-1}(\zero)$ and $\bb\in Q_j(\zero)$, as is done in Lemma~\ref{lem:structure-sheaf}.
\end{proof}

%%%%%%%%%%%%%%%%%%%%%%%%%%%%%%%%%%%%%%%%%%%%%%%%%%%%%%%%%%%%%%%%%%%%%%%%%%%%%%%%%%%%%

\pagebreak[1]
\section{Multigraded Regularity and Betti Numbers}\label{sec:betti}

Unlike in the single graded setting, it is possible for two modules on a product of projective spaces to have the same multigraded Betti numbers but different multigraded regularities.

\begin{example}\label{ex:same-betti}
  Let $M$ be the module on $\PP1\times\PP1$ with resolution
  \[ S(-1,0)^2\oplus S(0,-1)^2\gets S(-1,-1)^4\gets 0 \]
  given in Example~\ref{ex:not-linear}.  Computation shows that $M$ is $(1,0)$-regular but not $(0,1)$-regular. Notice that all of the twists appearing in the minimal free resolution of $M$ are symmetric with respect to the factors of $\PP1\times\PP1$\!. Hence the cokernel $N$ given by exchanging $x$ and $y$ in the presentation matrix has the same multigraded Betti numbers as $M$\!. However $N$ is not $(1,0)$-regular because $M$ was not $(0,1)$-regular.
\end{example}

\begin{remark}
	Example~\ref{ex:same-betti} answers a question of Botbol and Chardin \cite[Ques.~1.2]{BC17}.
\end{remark}

%%%%%%%%%%%%%%%%%%%%%%%%%%%%%%%%%%%%%%%%%%%%%%%%%%%%%%%%%%%%%%%%%%%%%%%%%%%%%%%%%%%%%

\subsection{Inner Bound from Betti Numbers}\label{sec:betti-inner-bounds}

While the multigraded Betti numbers of a module do not determine its regularity, in this section we show that they do determine a subset of the regularity.  In particular, the following lemma restricts the possible Betti numbers of a truncation of $M$ given the Betti numbers of $M$\!.  Intuitively, it states that the degrees of Betti numbers of $M_{\geq\dd}$ come from the maximum of $\dd$ and the degrees of Betti numbers of $M$\!, possibly after adding some linear terms.

\begin{lemma}\label{lem:truncation-betti}
	Let $M$ be a $\ZZ^r$-graded $S$-module.  If $M_{\geq\dd}$ has $\Tor_{m'}^S(M_{\geq\dd},\kk)_{\bb'}\neq 0$ for some $\bb'\in\ZZ^r$ then there exist $\bb\leq\bb'$ and $m\leq m'$ such that $\Tor_m^S(M,\kk)_\bb\neq 0$ and $|\bb'-\cc|\leq m'-m$ where $\cc=\max\{\bb,\dd\}$.
\end{lemma}
\begin{proof}
	Let $0 \gets M \gets F_0 \gets F_1 \gets \cdots$ be the minimal free resolution of $M$\!.  Then the termwise truncation
	\( 0 \gets M_{\geq\dd} \gets (F_0)_{\geq\dd} \gets (F_1)_{\geq\dd} \gets \cdots \)
	is also exact by Lemma~\ref{lem:truncation-functor}.  For each $i$, let $G^i_\b$ be a minimal free resolution of $(F_i)_{\geq\dd}$.
%https://q.uiver.app/?q=WzAsMTcsWzAsMywiMCJdLFsxLDMsIihGXzApX3tcXGdlcVxcZGR9Il0sWzIsMywiKEZfMSlfe1xcZ2VxXFxkZH0iXSxbMywzLCIoRl8yKV97XFxnZXFcXGRkfSJdLFsyLDIsIkdeMV8wIl0sWzIsMSwiR14xXzEiXSxbMSwyLCJHXjBfMCJdLFsxLDEsIkdeMF8xIl0sWzEsMCwiXFx2ZG90cyJdLFsxLDQsIjAiXSxbMiwwLCJcXHZkb3RzIl0sWzIsNCwiMCJdLFszLDIsIkdeMl8wIl0sWzMsMSwiR14yXzEiXSxbNCwzLCJcXGNkb3RzIl0sWzMsNCwiMCJdLFszLDAsIlxcdmRvdHMiXSxbMiwxXSxbMSwwXSxbMywyXSxbOCw3XSxbNyw2XSxbNiwxXSxbMSw5XSxbMTAsNV0sWzUsNF0sWzQsMl0sWzIsMTFdLFsxNCwzXSxbMTYsMTNdLFsxMywxMl0sWzEyLDNdLFszLDE1XV0=
  \[\begin{tikzcd}[column sep=small, row sep=small]
	& \vdots & \vdots & \vdots \\
	& {G^0_1} & {G^1_1} & {G^2_1} \\
	& {G^0_0} & {G^1_0} & {G^2_0} \\
	0 & {(F_0)_{\geq\dd}} & {(F_1)_{\geq\dd}} & {(F_2)_{\geq\dd}} & \cdots \\
	& 0 & 0 & 0
	\arrow[from=4-3, to=4-2]
	\arrow[from=4-2, to=4-1]
	\arrow[from=4-4, to=4-3]
	\arrow[from=1-2, to=2-2]
	\arrow[from=2-2, to=3-2]
	\arrow[from=3-2, to=4-2]
	\arrow[from=4-2, to=5-2]
	\arrow[from=1-3, to=2-3]
	\arrow[from=2-3, to=3-3]
	\arrow[from=3-3, to=4-3]
	\arrow[from=4-3, to=5-3]
	\arrow[from=4-5, to=4-4]
	\arrow[from=1-4, to=2-4]
	\arrow[from=2-4, to=3-4]
	\arrow[from=3-4, to=4-4]
	\arrow[from=4-4, to=5-4]
  \end{tikzcd}\]
	We will see in Corollary~\ref{cor:S-reg} that $S(-\bb)_{\geq\dd}$ has a linear resolution for all $\bb\in\ZZ^r$\!.  Thus the $G^i_\b$ are linear.  By taking iterated mapping cones we can construct a free resolution of $M_{\geq\dd}$ with terms
	\begin{align}\label{eq:nonmin-res}
		0 \gets G^0_0 \gets G^0_1 \oplus G^1_0 \gets G^0_2 \oplus G^1_1 \oplus G^2_0 \gets \cdots.
	\end{align}
	Then $\bb'$ corresponds to the degree of a generator of some $G^i_j$ with $i+j=m'$.  Since $G^i_\b$ is linear, there is a minimal generator of $(F_i)_{\geq\dd}$ with degree $\cc$ such that $|\bb'-\cc|=j$.

	However the generators of $(F_i)_{\geq\dd}$ have degrees equal to $\max\{\bb,\dd\}$ for degrees $\bb$ of generators of $F_i$.  These correspond to $\bb\in\ZZ^r$ such that $\Tor_i^S(M,\kk)_\bb\neq 0$.  Thus the lemma holds for $m=i$, so that $m'-m=j=|\bb'-\cc|$ as desired.
\end{proof}

Lemma~\ref{lem:truncation-betti} shows that each Betti number of $M_{\geq\dd}$ comes from a Betti number of $M$ in a predictable way.  Note that the process cannot be reversed---not all Betti numbers of $M$ produce minimal Betti numbers of $M_{\geq\dd}$.  However, the Betti numbers of $M$ limit the degrees where a nonlinear truncation could exist.  The following theorem identifies such degrees.
\begin{theorem}\label{thm:betti-truncations}
	Let $M$ be a $\ZZ^r$-graded $S$-module.  For all $\dd\in\bigcap L_m(\bb)$, the truncation $M_{\geq \dd}$ has a linear resolution generated in degree $\dd$, where the intersection is over all $m$ and all $\bb$ with $\Tor_m^S(M,\kk)_\bb\neq 0$.
\end{theorem}
\begin{proof}
	We may assume that $\dd=\zero$.  Suppose instead that $M_{\geq\zero}$ does not have a linear resolution generated in degree $\zero$.  Then there exist $\bb'\in\NN^r$ and $m'\in\ZZ$ such that $\Tor_{m'}^S(M_{\geq\zero},\kk)_{\bb'}\neq 0$ and $|\bb'|>m'$.

	By Lemma~\ref{lem:truncation-betti} there exist $\bb$ and $m$ so that $\Tor_m^S(M,\kk)_\bb\neq 0$ and $|\bb'-\cc|\leq m'-m$ where $\cc=\max\{\bb,\zero\}$.  The sum of the positive components of $\bb$ is
	\[|\cc|=|\bb'|-|\bb'-\cc|>m'-(m'-m)=m\]
	so $\zero\notin L_m(\bb)$ (see Remark~\ref{rem:describeL}).
\end{proof}

An analogous statement to Theorem~\ref{thm:betti-truncations} exists for truncations with quasilinear resolutions.  By Theorem~\ref{thm:quasilinear-conjecture} it also gives a subset of the multigraded regularity.  We will see in Section~\ref{sec:complete-intersections} that this inner bound is sharp.
\begin{theorem}\label{thm:betti-regularity}
	Let $M$ be a $\ZZ^r$-graded $S$-module.  For all $\dd\in\bigcap Q_m(\bb)$, the truncation $M_{\geq \dd}$ has a quasilinear resolution generated in degree $\dd$, where the intersection is over all $m$ and all $\bb$ with $\Tor_m^S(M,\kk)_\bb\neq 0$.
\end{theorem}
\begin{proof}
	Assume $\dd=\zero$ and suppose instead that $M_{\geq\zero}$ does not have a quasilinear resolution generated in degree $\zero$.  If $M_{\geq\zero}$ is not generated in degree $\zero$ then some generator of $M$ has a degree $\bb$ with a positive coordinate, so that $\zero\notin\bb+\NN^r=Q_0(\bb)$.

	Otherwise there exist $\bb'\in\NN^r$ and $m'\in\ZZ$ such that $\Tor_{m'}^S(M_{\geq\zero},\kk)_{\bb'}\neq 0$ and $|\bb'|>m'+\ell'-1$ where $\ell'$ is the number of nonzero coordinates in $\bb'$.  Thus by Lemma~\ref{lem:truncation-betti} there exist $\bb$ and $m$ so that $\Tor_m^S(M,\kk)_\bb\neq 0$ and $|\bb'-\cc|\leq m'-m$ for $\cc=\max\{\bb,\zero\}$.

	Let $\ell$ be the number of coordinates for which $\cc$ differs from $\cc'=\max\{\bb,\one\}$.  Then $|\cc'|=|\cc|+\ell$, so the sum of the positive components of $\bb-\one$ is
	\begin{align*}
	|\cc'-\one| &= |\cc|+\ell-r\\
	&= |\bb'|-|\bb'-\cc|-r+\ell\\
	&> (m'+\ell'-1)-(m'-m)-r+\ell\\
	&= m-1+\ell'-(r-\ell).
	\end{align*}
	Note that $r-\ell$ is the number of nonzero coordinates in $\cc$.  Since $\bb'\geq\zero$ and $\bb'\geq\bb$ we have $\bb'\geq\cc\geq\zero$, so $\ell'\geq r-\ell$.  Hence the right side of the inequality is $\geq m-1$, so $\zero\notin L_{m-1}(\bb-\one) = Q_m(\bb)$ (see Remark~\ref{rem:describeL}).
\end{proof}

\begin{corollary}\label{cor:betti-reg-bound}
  Let $M$ be a finitely generated $\ZZ^r$-graded $S$-module. If $\HH_B^0(M)=0$, then
  \[ \bigcap_{i\in\NN}\bigcap_{\bb\in\beta_i(M)} Q_i(\bb) \subseteq \reg(M). \]
\end{corollary}

We can now prove Proposition~\ref{prop:invariance}.

\begin{proof}[Proof of Proposition \ref{prop:invariance}]
	Suppose that $M_{\geq\dd}$ has a linear resolution.  We will apply Theorem~\ref{thm:betti-truncations} to $M_{\geq\dd}$ to show that $M_{\geq\dd'}$ has a linear resolution for $\dd'\geq\dd$ as desired.  We may assume that the intersection contains all possible terms that could arise from a linear resolution:
	\[ \bigcap_{i\in\NN}\bigcap_{-\bb\in L_i(-\dd)} L_i(\bb) \]
	Note that $-\bb\in L_i(-\dd)$ if and only if $\dd\in L_i(\bb)$.  Thus $\dd\in L_i(\bb)$ for all $\bb$, so $\dd'$ is in the intersection as well.  For quasilinear resolutions replace $L$ with $Q$ and apply Theorem~\ref{thm:betti-regularity} .
\end{proof}

Other bounds on the multigraded regularity of a module in terms of its Betti numbers exist in the literature. For example, Maclagan and Smith use a long exact sequence argument to bound regularity in \cite[Thm.~1.5, Cor.~7.2]{MS04}. While our theorem has the added hypothesis that $\HH_B^0(M)=0$, it is often sharper than Maclagan and Smith's.

\begin{example}\label{ex:betti-bound-ms}
  In \cite[Ex.~7.6]{MS04} Maclagan and Smith consider the $B$-saturated ideal $I = \langle x_{1,0}-x_{1,1},x_{2,0}-x_{2,1},x_{3,0}-x_{3,1}\rangle \cap \langle x_{1,0}-2x_{1,1},x_{2,0}-2x_{2,1},x_{3,0}-2x_{3,1}\rangle$ on $\PP1\times\PP1\times\PP1$\!. They show that the regularity of $S/I$ is
  \[ \reg(S/I) = \left((1,0,0)+\NN^3\right) \cup \left((0,1,0)+\NN^3\right) \cup \left((0,0,1)+\NN^3\right) \]
  and their bound from the Betti numbers of $S/I$ is
  \[ \left((2,2,1)+\NN^3\right) \cup \left((2,1,2)+\NN^3\right) \cup \left((1,2,2)+\NN^3\right)\subset \reg(S/I). \]
  However, Corollary~\ref{cor:betti-reg-bound} implies that $(1,1,1)+\NN^r\subset\reg(S/I)$, giving a larger inner bound.
\end{example}

%%%%%%%%%%%%%%%%%%%%%%%%%%%%%%%%%%%%%%%%%%%%%%%%%%%%%%%%%%%%%%%%%%%%%%%%%%%%%%%%%%%%%

\subsection{Regularity of Complete Intersections}\label{sec:complete-intersections}

As an application of Theorems~\ref{main-thm-reg} and \ref{main-thm-betti}, in this section we compute the multigraded regularity of a saturated complete intersection satisfying minor hypotheses on its generators. To do this we make the bound from Corollary~\ref{cor:betti-reg-bound} explicit in the case of complete intersections. We then use our characterization of regularity to prove that the resulting bound is sharp by explicitly constructing truncations outside this region that do not have quasilinear resolutions.

\begin{lemma}\label{lem:Q-inclusions}
  If $\bb,\cc\in\NN^r$ with $b_j,c_j>0$ for all $j$ then $Q_{i+1}(\bb+\cc)\subseteq Q_i(\bb)$ for all $i>0$.
\end{lemma}
\begin{proof}
  By definition the minimal elements of $Q_{i+1}(\bb+\cc)$ are of the form $\bb+\cc-\one-\vv$ where $\vv\in\NN^r$ and $|\vv|=i$.  It is enough to show that each $\bb+\cc-\one-\vv$ is in $Q_i(\bb)$. Since $|\vv|=i$ it has at least one nonzero coordinate, say $v_j$.  From this we have
  \[ \bb+\cc-\one-\vv=\left(\bb-\one-\left(\vv-\ee_j\right) \right)+(\cc-\ee_j). \]
  The desired containment follows from the above equality given that $|\vv-\ee_j|=i-1$ and that by assumption $\cc-\ee_j$ is in $\NN^r$\!.
\end{proof}

%\mahrud{Research: when is a CI on a product of projective spaces saturated?}
\begin{theorem}\label{thm:ci-regularity}
	Let $I = \langle f_1,\ldots,f_c\rangle\subset B$ be a saturated complete intersection of codimension $c$ in $S$, meaning that the $f_i$ form a regular sequence of elements from $B$ and $H_B^0(S/I) = 0$. Then
	\[ \reg(S/I) = Q_c\left(\sum_{i=1}^c\deg f_i\right). \]
\end{theorem}
\begin{proof}
	Write $\aa = \sum_{i=1}^c\deg f_i$.  By Theorem~\ref{thm:quasilinear-conjecture} it suffices to show that $(S/I)_{\geq\dd}$ has a quasilinear resolution generated in degree $\dd$ if and only if $\dd\in Q_c(\aa)$.  We will prove one direction by showing that $Q_c(\aa)$ is the bound from Corollary~\ref{cor:betti-reg-bound}, i.e., that
	\[ \bigcap_{j\in\NN} \bigcap_{\bb\in\beta_j(S/I)} Q_j(\bb)=Q_c(\aa) \]
	By hypothesis the minimal free resolution $F_\b$ of $S/I$ is a Koszul complex, so the elements of $\beta_j(S/I)$ are sums of $j$ choices of $\deg f_i$.  In particular $\beta_0(S/I)=\{\zero\}$ and $\beta_c(S/I)=\{\aa\}$.  We have $Q_c(\aa)\subset\NN^r= Q_0(\zero)$, so it suffices to show that
  \[ Q_{j+1}(\deg f_{i_1}+\cdots+\deg f_{i_j}+\deg f_{i_{j+1}})\subseteq Q_j(\deg f_{i_1}+\cdots+\deg f_{i_j}) \]
	for all $0<j<c$ and all $1\leq i_1<\cdots<i_{j+1}\leq c$, since each of the other sets in the intersection can be obtained from $Q_c(\aa)$ in this way.  Note that since $I\subset B$, all coordinates of each $\deg f_i$ are positive; therefore the inclusion follows from Lemma~\ref{lem:Q-inclusions}.

	Now we need that $(S/I)_{\geq\dd}$ does not have a quasilinear resolution if $\dd\notin Q_c(\aa)$.  Specifically, we will show that the resolution of $(S/I)_{\geq\dd}$ has a $c$-th syzygy in degree $\aa'=\max\{\dd,\aa\}$.  If $\dd\notin Q_c(\aa)$ then $\dd\notin Q_c(\aa')$ and thus $-\aa'\notin Q_c(-\dd)$, so this will complete our argument.

	The proof of Lemma~\ref{lem:truncation-betti} constructs a possibly nonminimal free resolution \eqref{eq:nonmin-res} of $(S/I)_{\geq\dd}$ from resolutions of truncations of the $F_j$.  Since $(F_c)_{\geq\dd}$ has a generator of degree $\aa'$, the minimal free resolution of $(S/I)_{\geq\dd}$ will contain a $c$-th syzygy of degree $\aa'$ unless there is a nonminimal map from the generators $G_0^c$ of $(F_c)_{\geq\dd}$ to $G_0^{c-1}\oplus\cdots\oplus G_{c-1}^0$.  Suppose for contradiction that this is true.

	The degrees of the summands in $G_i^{c-1-i}$ have the form $\max\{\dd,\bb\}+\vv$ where $\bb$ is the sum of the degrees of $c-1-i$ choices of the generators $f_j$ and some $\vv\in\NN^r$ with $|\vv|=i$.  In order to have a degree $\zero$ map we need $\max\{\dd,\bb\}+\vv=\aa'=\max\{\dd,\aa\}$ for some $\bb$ and $\vv$.  Since all coordinates of each $\deg f_j$ are positive $b_j+i+1\leq a_j$ for each $j$, so $b_j+v_j\neq a_j$.  Thus $\dd\geq\bb$, so $\dd+\vv=\aa'$, contradicting the fact that $\dd\notin Q_c(\aa')$.
\end{proof}

%\mahrud{Research: extend the theorem to cover Example~\ref{ex:betti-bound-ms}}

Note the assumption that $H_B^0(S/I)=0$ is automatically satisfied if $\codim(P)\neq\codim(I)$ for all minimal primes $P$ over $B$. However, based on a number of examples it seems that a weaker saturation hypothesis may be sufficient.

\begin{example}
  Write $S = \kk[x_0,x_1,x_2,y_0,y_1,y_2]$ and consider the saturated complete intersection ideal $I = \langle x_0 y_0, x_1 y_1^2 \rangle$ that defines a surface in $\PP2\times\PP2$\!. Then Theorem~\ref{thm:ci-regularity} implies
  \[ \reg(S/I) = Q_2\big((2,3)\big) = \left((0,2)+\NN^2\right) \cup \left((1,1)+\NN^2\right). \]
\end{example}

An interesting application of Theorem~\ref{thm:ci-regularity} concerns complete intersection Calabi--Yau subvarieties of products of projective spaces, which by definition all satisfy the equality $\sum_{i=1}^c\deg f_i = \nn + \one = (n_1+1,\dots,n_r+1)$ \cite[\S2]{GHL13}.

\begin{corollary}\label{cor:ci-cy}
  Let $X\subset\PP\nn$ be a complete intersection Calabi--Yau variety of codimension $c$.\linebreak
  If the saturated defining ideal of $X$ is contained in $B$, then
  \[ \reg(S/I) = Q_c\left(\nn + \one\right). \]
\end{corollary}

%%%%%%%%%%%%%%%%%%%%%%%%%%%%%%%%%%%%%%%%%%%%%%%%%%%%%%%%%%%%%%%%%%%%%%%%%%%%%%%%%%%%%

\section{Linear Truncations}\label{sec:linear-trun}

As demonstrated by Example~\ref{ex:not-linear}, in general $\dd$-regularity is a weaker condition than having a linear resolution for $M_{\geq\dd}$. Still, linear truncations have been independently studied in the literature \cite{EES15,BES20}.

Our main result in this section is a cohomological vanishing condition that specifies when $M_{\geq\dd}$ has a linear free resolution. Our arguments largely mimic those for the analogous statements about quasilinear free resolutions by switching the roles of $L$ and $Q$. Indeed, if one cares specifically about linear syzygies then this vanishing condition could just as easily be taken as the definition of regularity on a product of projective spaces.

\begin{lemma}\label{lem:BES-sreg-to-lin}
	Let $M$ be a $\ZZ^r$-graded $S$-module.  If $\HH^i(\PP\nn\!, \shM(\bb))=0$ for all $i>0$ and all $\bb\in Q_i(\zero)$, then $\HH^{|\aa|-i}(\PP\nn\!, \shM\otimes\Om_{\PP\nn}^\aa(\aa))=0$ for all $i\geq 0$ and all $-\aa\notin L_i(\zero)$.
\end{lemma}
\begin{proof}
	We will modify the argument from Lemma~\ref{lem:BES-reg-to-qlin}.

	Suppose that $-\aa\notin L_i(\zero)$ and $\HH^{|\aa|-i}(\PP\nn\!, \shM\otimes\Om_{\PP\nn}^\aa(\aa))\neq 0$.  Since $\aa\geq 0$ we have $|\aa|>i$.  There must exist $j$ such that $\HH^{|\aa|-i+j}(\PP\nn\!, \shM\otimes\cF_j)\neq 0$, where the twists $\bb$ in $F_j$ have total degree $-j-\ell$ for $\ell$ the number of nonzero coordinates in $\aa$.  Each twist has $\ell$ negative coordinates, so that the positive coordinates of $-\one-\bb$ sum to $j+\ell-\ell=j$.  Hence $\HH^{|\aa|-i+j}(\PP\nn\!, \shM(\bb))\neq 0$ for some $\bb\in L_j(-\one)= Q_{j+1}(\zero)\subseteq Q_{|\aa|-i+j}(\zero)$ with $|\aa|-i+j>0$.
\end{proof}

As in our main theorem, the conclusion of this lemma ensures the vanishing of certain Betti numbers of $M_{\geq\dd}$.

\begin{theorem}\label{thm:sreg-iff-lin}
  Let $M$ be a finitely generated $\ZZ^r$-graded $S$-module with $\HH_B^0(M) = 0$.  Then $M_{\geq\dd}$ has a linear resolution $F_{\b}$ with $F_{0}$ generated in degree $\dd$ if and only if $\HH_B^i(M)_\bb=0$ for all $i>0$ and all $\bb\in Q_{i-1}(\dd)$.
\end{theorem}
\begin{proof}
	As in the proof of Theorem~\ref{thm:qlin-to-reg}, we may assume $\dd=\zero$.  The proof of the forward implication is analogous to the argument there, switching the roles of $L$ and $Q$.  For the reverse, notice that the proof of Theorem~\ref{thm:reg-to-qlin} shows that the virtual resolution of $M$ from Proposition~\ref{prop:free-monad} has the same Betti numbers as the minimal free resolution of $M_{\geq\dd}$, i.e.,
  \[ \dim_\kk\Tor_j^S(M_{\geq\dd}, \kk)_\aa = \dim_\kk\HH^{|\aa|-j}(\PP\nn\!, \shM\otimes\Om_{\PP\nn}^\aa(\aa)) \]
  for $|\aa|\geq j\geq 0$ and both are 0 otherwise. The vanishing of the right hand side for $-\aa\notin L_j(\zero)$, given by Lemma~\ref{lem:BES-sreg-to-lin}, then implies that the minimal free resolution of $M_{\geq\dd}$ is linear.
\end{proof}

\begin{corollary}\label{cor:S-reg}
	The minimal free resolution of $S(-\bb)_{\geq\dd}$ is linear for all $\bb,\dd\in\ZZ^r$\!.
\end{corollary}
\begin{proof}
	By adjusting $\dd$ we may assume that $\bb=\zero$.  Note that $S_{\geq\dd}=S_{\geq\dd'}$ for $\dd'=\max\{\dd,\zero\}\in\zero+\NN^r$\!.  Thus by Theorem~\ref{thm:sreg-iff-lin} and Proposition~\ref{prop:invariance} it suffices to show that $H_B^i(S)_\bb=0$ for all $i>0$ and all $\bb\in Q_{i-1}(\zero)$, which follows from Lemma~\ref{lem:structure-sheaf}.
\end{proof}

\section{Generalizing Eisenbud--Goto}\label{sec:regularity-regions}

Recall Eisenbud--Goto's conditions \eqref{item:local-coh} through \eqref{item:single-betti} from the introduction. As we have seen, these conditions diverge substantially for products of projective spaces. However, they can each be generalized to give interesting, albeit different, regions inside $\Pic\PP\nn$\!.

If $M$ is a finitely generated $\ZZ^r$-graded $S$-module, then \eqref{item:local-coh} defines the multigraded regularity region $\reg(M)\subseteq\Pic\PP\nn$ of Maclagan and Smith.  On the other hand condition \eqref{item:single-lin} naturally generalizes to two truncation regions, depending on whether one wishes to preserve the degree of maps or the equivalence with \eqref{item:local-coh}. First, the obvious generalization gives the linear truncation region:
\[ \trunc^L(M) \coloneqq \left\{\dd\in\ZZ^r \;\;\middle|\;\; M_{\geq\dd}\text{ has a linear free resolution generated in degree $\dd$}\right\}. \]
Second, our characterization of regularity gives the quasilinear truncation region:
\[ \trunc^Q(M)\coloneqq \left\{\dd\in\ZZ^r \;\;\middle|\;\; M_{\geq\dd}\text{ has a quasilinear free resolution generated in degree $\dd$}\right\}. \]
Finally, condition \eqref{item:single-betti} on the Betti numbers of $M$ also naturally generalizes to two Betti regions, the $L$-Betti region as in Theorem~\ref{thm:betti-truncations} and the $Q$-Betti region as in Theorem~\ref{thm:betti-regularity}:
\begin{align*}
  \betti^L(M) \coloneqq \bigcap_{i\in\NN}\bigcap_{\dd\in \beta_i(M)} L_i(\dd), &&&&
  \betti^Q(M) \coloneqq \bigcap_{i\in\NN}\bigcap_{\dd\in \beta_i(M)} Q_i(\dd).
\end{align*}

Theorem~\ref{thm:quasilinear-conjecture} now states that $\reg(M) = \trunc^Q(M)$ when $\HH_B^0(M)=0$. Moreover, since all linear resolutions are quasilinear we get $\trunc^L(M)\subseteq\trunc^Q(M)$. Similarly, since $L_i(\dd)\subseteq Q_i(\dd)$, by definition $\betti^L(M)\subseteq\betti^Q(M)$.

Theorem~\ref{thm:betti-truncations} shows that the $L$-Betti region $\betti^L(M)$ is a subset of the linear truncation region $\trunc^L(M)$. Similarly, Theorem~\ref{thm:betti-regularity} shows that the $Q$-Betti region $\betti^Q(M)$ is a subset of the quasilinear truncation region $\trunc^Q(M)$. We can summarize all of the above relations in the following highly non-commutative diagram:

\[\begin{tikzcd}[row sep = 2em, column sep = 3em]
  \betti^L(M) \rar[hookrightarrow]{\ref{thm:betti-truncations}} \dar[hookrightarrow] &
  \trunc^L(M) \dar[hookrightarrow]{} & \\
  \betti^Q(M) \rar[hookrightarrow]{\ref{thm:betti-regularity}} &
  \trunc^Q(M) \rar[equals]{\ref{thm:quasilinear-conjecture}} & \reg(M)
\end{tikzcd}\]

We saw in Section~\ref{sec:linear-trun} that we can switch the roles of $Q$ and $L$ in the proof of Theorem~\ref{thm:quasilinear-conjecture} to complete the upper right corner of this diagram. The resulting cohomological characterization of $\trunc^L(M)$ in Theorem~\ref{thm:sreg-iff-lin} is related to the positivity conditions described in Remark~\ref{rem:BES-1.2}. We suspect that the reversal of $Q$ and $L$ between the Betti number and cohomological conditions has a deeper explanation in terms of the BGG correspondence.

We illustrate the four regions above in the following example.

\begin{example}
  Let $I$ be the $B$-saturated ideal in Example~\ref{ex:hyperelliptic-curve}, defining a smooth hyperelliptic curve of genus 4 embedded into $\PP1\times\PP2$ as a curve of degree $(2,8)$. As noted in \cite[Ex.~1.4]{BES20}, using \textit{Macaulay2} one finds that the minimal graded free resolution of $I$ is:
  \[\begin{tikzcd}
  S & \lar
  \begin{matrix}
    S(-3,-1)\\[-3pt]
    \oplus\\[-3pt]
    S(-2,-2)\\[-3pt]
    \oplus\\[-3pt]
    S(-2,-3)^2\\[-3pt]
    \oplus\\[-3pt]
    S(-1,-5)^3\\[-3pt]
    \oplus\\[-3pt]
    S(0,-8)
  \end{matrix}
  & \lar
  \begin{matrix}
    S(-3,-3)^3\\[-3pt]
    \oplus\\[-3pt]
    S(-2,-5)^6\\[-3pt]
    \oplus\\[-3pt]
    S(-1,-7)\\[-3pt]
    \oplus\\[-3pt]
    S(-1,-8)^2
  \end{matrix}
  & \lar
  \begin{matrix}
    S(-3,-5)^3\\[-3pt]
    \oplus \\[-3pt]
    S(-2,-7)^2\\[-3pt]
    \oplus \\[-3pt]
    S(-2,-8)
  \end{matrix}
  & \lar S(-3,-7)
  & \lar 0.
  \end{tikzcd}\]
  From this we can calculate that $\betti^L(S/I)$ and $\betti^Q(S/I)$ are both equal to $(2,7)+\NN^2$\!.  These regions, depicted in Figure~\ref{fig:four-regions-landscaping}, can also be computed using \texttt{linearTruncationsBound} and \texttt{regularityBound} from the \textit{Macaulay2} package \texttt{LinearTruncations}, which implement Theorems~\ref{thm:betti-truncations} and \ref{thm:betti-regularity}, respectively \cite{CHN22}.

  Further, using the functions \texttt{linearTruncations} and \texttt{multigradedRegularity} from the package \texttt{VirtualResolutions} \cite{ABLS20}, we can compute where $S/I$ has a linear or quasilinear truncation inside the box $[0,9]^2$. We see that the minimal elements of $\trunc^L(S/I)$ are $(1,5), (2,2)$, and $(5,1)$. On the other hand the minimal elements of $\trunc^Q(S/I)$---which equals $\reg(S/I)$ as $I$ is saturated---are $(1,5), (2,2),$ and $(4,1)$.
\end{example}

\begin{figure}[h]
  \newcommand{\makegrid}{
  \path[use as bounding box] (-1.5,-2) rectangle (9.5,9.2);
  \foreach \x in {-1,...,9}
  \foreach \y in {-1,...,9}
    { \fill[gray,fill=gray] (\x,\y) circle (1.5pt); }
  \draw[-,  semithick] (-1,0)--(9,0);
  \draw[-,  semithick] (0,-1)--(0,9);
}

\begin{tikzpicture}[scale=.3]
  \path[fill=PineGreen!45] (1,5)--(1,9)--(9,9)--(9,1)--(5,1)--(5,1)--(5,2)--(2,2)--(2,5)--(1,5);
  \makegrid
  \draw[->, ultra thick,PineGreen] (5,1)--(9,1);
  \draw[-, cap=round,ultra thick,PineGreen] (5,1)--(5,2)--(2,2)--(2,5)--(1,5);
  \draw[->, ultra thick,PineGreen] (1,5)--(1,9);
  \fill[TealBlue,fill=PineGreen] (2,2) circle (6pt);
  \fill[TealBlue,fill=PineGreen] (5,1) circle (6pt);
  \fill[TealBlue,fill=PineGreen] (1,5) circle (6pt);
  \node at (4.5,-2.5) {{\footnotesize $\trunc^L(S/I)$}};
\end{tikzpicture}
%%%%%%%%%%%%%%%%%%%%%%%%%%%%%%%%%%%%
\quad %%%%%%%%%%%%%%%%%%%%%%%%%%%%%%
%%%%%%%%%%%%%%%%%%%%%%%%%%%%%%%%%%%%
\begin{tikzpicture}[scale=.3]
  \path[fill=PineGreen!45] (9,9)--(2,9)--(2,7)--(9,7)--(9,9);
  \makegrid
  \draw[->, ultra thick,PineGreen] (2,7)--(2,9);
  \draw[->, ultra thick,PineGreen] (2,7)--(9,7);
  \fill[TealBlue,fill=PineGreen] (2,7) circle (6pt);
  \node at (4.5,-2.5) {{\footnotesize $\betti^L(S/I)$}};
\end{tikzpicture}
%%%%%%%%%%%%%%%%%%%%%%%%%%%%%%%%%%%%
\quad %%%%%%%%%%%%%%%%%%%%%%%%%%%%%%
%%%%%%%%%%%%%%%%%%%%%%%%%%%%%%%%%%%%
\begin{tikzpicture}[scale=.3]
  \path[fill=DarkOrchid!45] (9,9)--(2,9)--(2,7)--(9,7)--(9,9);
  \makegrid
  \draw[->, ultra thick,DarkOrchid] (2,7)--(2,9);
  \draw[->, ultra thick,DarkOrchid] (2,7)--(9,7);
  \fill[TealBlue,fill=DarkOrchid] (2,7) circle (6pt);
  \node at (4.5,-2.5) {{\footnotesize $\betti^Q(S/I)$}};
\end{tikzpicture}
%%%%%%%%%%%%%%%%%%%%%%%%%%%%%%%%%%%%
\quad %%%%%%%%%%%%%%%%%%%%%%%%%%%%%%
%%%%%%%%%%%%%%%%%%%%%%%%%%%%%%%%%%%%
\begin{tikzpicture}[scale=.3]
  \path[fill=DarkOrchid!45] (9,9)--(9,1)--(4,1)--(4,2)--(2,2)--(2,5)--(1,5)--(1,9)--(9,9);
  \makegrid
  \draw[->, ultra thick,DarkOrchid] (4,1)--(9,1);
  \draw[-, cap=round,ultra thick,DarkOrchid] (4,1)--(4,2)--(2,2)--(2,5)--(1,5);
  \draw[->, ultra thick,DarkOrchid] (1,5)--(1,9);
  \fill[TealBlue,fill=DarkOrchid] (2,2) circle (6pt);
  \fill[TealBlue,fill=DarkOrchid] (4,1) circle (6pt);
  \fill[TealBlue,fill=DarkOrchid] (1,5) circle (6pt);
  \node at (4.5,-2.5) {{\footnotesize $\trunc^Q(S/I) = \reg(S/I)$}};
\end{tikzpicture}
  \caption{The four regions for Example~\ref{ex:hyperelliptic-curve} inside $\Pic\left(\PP1\times\PP2\right)$.}\label{fig:four-regions-landscaping}
\end{figure}

\begin{remark}
  Beyond products of projective spaces, the cones of nef and effective divisors on an arbitrary smooth projective toric variety no longer coincide in general. While Maclagan and Smith's definition of multigraded Castelnuovo--Mumford regularity still applies, this divergence necessitates new definitions of linearity for graded complexes and truncation for graded modules.
\end{remark}

\vfill

%%%%%%%%%%%%%%%%%%%%%%%%%%%%%%%%%%%%%%%%%%%%%%%%%%%%%%%%%%%%%%%%%%%%%%%%%%%%%%%%%%%%%
\pagebreak[4]

\begin{appendix}

\section{Virtual Resolutions via Fourier--Mukai Transforms}\label{sec:appendix}

Fourier--Mukai transforms, the Beilinson spectral sequence, and tilting functors are tools from derived categories and representation theory which are used in the proofs of Proposition~\ref{prop:free-monad} and Theorem~\ref{thm:uniqueness}. Here we introduce these concepts and present the details necessary for constructing the free monad in Proposition~\ref{prop:free-monad}, which is homotopic to the output of a certain Fourier--Mukai transform, and proving its uniqueness.

\subsection{The Fourier--Mukai Transform}\label{sec:fourier-mukai}

In this section we introduce a type of geometric functor between derived categories known as a Fourier--Mukai transform. See \cite[\S5]{Huybrechts2006} for background and further details.

Let $X$ and $Y$ be smooth projective varieties and consider the two projections
\[\begin{tikzcd}[column sep=small, row sep=tiny]
& {X\times Y} \\ X && Y.
\arrow["p", from=1-2, to=2-3]
\arrow["q"', from=1-2, to=2-1]
\end{tikzcd}\]
Given an object $\cK\in\Db{X\times Y}$, the associated \emph{Fourier--Mukai transform} is a functor \[ \Phi_\cK\colon \Db X\to \Db Y \] between the derived categories of bounded complexes of coherent sheaves. It is constructed as a composition of derived functors
\[ \cF \mapsto \R p_*\!\left(\L q^*\cF \Lotimes \cK\right), \]
and in fact all equivalences between $\Db X$ and $\Db Y$ arise in this way.  Here $\L q^*$, $\R p_*$, and $-\Lotimes\cK$ are the derived functors induced by $q^*$, $p_*$, and $-\otimes\cK$, respectively. Moreover, since $q$ is flat $\L q^*$ is the usual pull-back, and if $\cK$ is a complex of locally free sheaves $-\Lotimes\cK$ is the usual tensor product.

A special case of the Fourier--Mukai transform occurs when $Y = X$ and $\cK\in\Db{X\times X}$ is a resolution of the structure sheaf $\OO_\Delta$ of the diagonal subscheme $\iota\colon\Delta\to X\times X$\!.  Such $\cK$ is referred to as a \textit{resolution of the diagonal}.

Using the projection formula, one can see that the Fourier--Mukai transform $\Phi_{\OO_\Delta}$ is simply the identity in the derived category; that is to say, replacing $\OO_\Delta$ with $\cK$ produces quasi-isomorphisms.

%%%%%%%%%%%%%%%%%%%%%%%%%%%%%%%%%%%%%%%%%%%%%%%%%%%%%%%%%%%%%%%%%%%%%%%%%%%%%%%%%%%%%

\subsection{The Beilinson Spectral Sequence}\label{sec:beilinson-ss}

Returning to the case of products of projective spaces, we consider coherent sheaves on $X=\PP\nn$\!.  We construct a free monad for $\shM$ from the Beilinson spectral sequence on $\PP\nn\times\PP\nn$ and describe its Betti numbers.  When $M$ is $\zero$-regular, this free monad is a virtual resolution, which is used in Section~\ref{sec:regularity-criterion}. See \cite[\S3.1]{OSS1980} for a geometric exposition and \cite[\S8.3]{Huybrechts2006} or \cite[\S3]{AO89} for an algebraic exposition on a single projective space.

For sheaves $\cF$ and $\cG$ on $\PP\nn$\!, denote $p^*\cF\otimes q^*\cG$ by $\cF\boxtimes\cG$.  Consider the vector bundle
\begin{align*}
  \cW &= \bigoplus_{i=1}^r \OO_{\PP\nn}(\ee_i) \boxtimes \cT_{\PP\nn}^{\ee_i}(-\ee_i),
\end{align*}
where $\cT_{\PP\nn}^{\ee_i}$ is the pullback of the tangent bundle $\cT_{\PP{n_i}}$, as in the Euler sequence on the single factor $\PP{n_i}$:
\vspace*{-0.5em}
\begin{equation*}\label{eq:euler-seq}
  \begin{tikzcd}
    0 & \cT_{\PP{n_i}} \lar & \OO_{\PP{n_i}}^{n_i+1}(1) \lar & \OO_{\PP{n_i}} \lar & 0 \lar.
  \end{tikzcd}
\end{equation*}

There is a canonical section $s\in\HH^0(\PP\nn\times\PP\nn\!,\cW)$ whose vanishing cuts out the diagonal subscheme $\Delta\subset\PP\nn\times\PP\nn$ (see \cite[Lem.~2.1]{BES20}), giving a Koszul resolution of $\OO_\Delta$:
\begin{equation*}\label{eq:res-o'diagonal}
  \begin{tikzcd}
\cK\colon 0 &\lar \OO_{\PP\nn \times \PP\nn} &\lar \cW^\vee &\lar \Alt^2\cW^\vee &\lar \cdots &\lar \Alt^n\cW^\vee &\lar 0.
  \end{tikzcd}
\end{equation*}

The terms of $\cK$ can be written as
\begin{align}\label{eq:res-o'diag-terms}
  \cK_j
  = \Alt^j \left(\bigoplus_{i=1}^r \OO_{\PP\nn}(-\ee_i)\boxtimes \Om_{\PP\nn}^{\ee_i}(\ee_i)\right)
  = \bigoplus_{|\aa|=j}            \OO_{\PP\nn}(-\aa)  \boxtimes \Om_{\PP\nn}^{\aa}(\aa), \quad \text{ for } 0\leq j\leq|\nn|.
\end{align}
As in Section~\ref{sec:fourier-mukai}, we are interested in the derived pushforward of $q^*\shM\otimes\cK$, which we will compute by resolving the second term of each box product with a \v{C}ech complex to obtain a spectral sequence.  Since $\cK$ is a resolution of the diagonal, the pushforward will be quasi-isomorphic to $\shM$\!.

Consider the double complex
\[ C^{-s,t} = \bigoplus_{|\aa|=s} \OO_{\PP\nn}(-\aa) \boxtimes \cech^t\Big(\mathfrak U_B, \shM\otimes\Om_{\PP\nn}^{\aa}(\aa)\Big), \]
with vertical maps from the \v{C}ech complexes and horizontal maps from $\cK$.  Since taking \v{C}ech complexes is functorial and exact we have \(\Tot(C) \sim q^*\shM\otimes\cK\), which is a resolution of \(q^*\shM\otimes\OO_\Delta \) because $\cK$ is locally free. Moreover, since the first term of each box product in $q^*\shM\otimes\cK$ is locally free, the columns of $C$ are $p_*$-acyclic (cf.\ \cite[Prop.~3.2]{Hartshorne1966}, \cite[Lem.~3.2]{AO89}).  Hence the pushforward
\begin{align}\label{eq:beilinson-E0}
  E_0^{-s,t} = p_*(C^{-s,t}) = \bigoplus_{|\aa|=s} \OO_{\PP\nn}(-\aa) \otimes \Gamma\left(\PP\nn\!, \cech^t\Big(\mathfrak U_B, \shM\otimes\Om_{\PP\nn}^{\aa}(\aa)\Big)\right)
\end{align}
satisfies \( \Tot(E_0) = \Phi_\cK(\shM) \sim \shM \). With this notation, the \emph{Beilinson spectral sequence} is the spectral sequence of the double complex $E_0$, whose first page after taking homology vertically has terms
\begin{align}\label{eq:beilinson-E1}
  E_1^{-s,t} = \bigoplus_{|\aa|=s} \OO_{\PP\nn}(-\aa)\otimes\HH^t\left(\PP\nn\!, \shM\otimes\Om_{\PP\nn}^\aa(\aa)\right) = \R^t p_*(q^*\shM\otimes\cK_s).
\end{align}

Beilinson's resolution of the diagonal and the associated spectral sequence are crucial ingredients in constructions of Beilinson monads, Tate resolutions, and virtual resolutions \cite{EFS03,EES15,BES20}. Recently, Brown and Erman \cite{BE26} expanded these constructions to toric varieties using a noncommutative analogue of a Fourier--Mukai transform. More generally, Costa and Mir\'{o}-Roig \cite{CMR07} have considered Beilinson-type spectral sequences for smooth projective varieties under certain conditions on the derived category.

%%%%%%%%%%%%%%%%%%%%%%%%%%%%%%%%%%%%%%%%%%%%%%%%%%%%%%%%%%%%%%%%%%%%%%%%%%%%%%%%%%%%%

\subsection{Beilinson's Collection and its Endomorphism Algebra}\label{sec:beilinson-quiver}

An important property of the Beilinson spectral sequence is that only the line bundles $\OO_{\PP\nn}(-\aa)$ for $\zero\leq\aa\leq\nn$ appear in \eqref{eq:beilinson-E1}. In particular, these line bundles form a full strong exceptional collection for $\Db{\PP\nn}$ \cite[Def.~4.1]{King97}, which induces an explicit equivalence, due to Bondal, between $\Db{\PP\nn}$ and the derived category of right modules over the path algebra of a bound quiver. This section briefly introduces this well-known equivalence, which is used in Section~\ref{sec:virtual-res} to prove uniqueness of certain virtual resolutions. For a detailed account of the representation theory of finite-dimensional algebras, see \cite{Bondal89,Bondal90,DW2017}.

A quiver $Q$ is a directed graph, allowing loops or repeated arrows. A path in $Q$ is a sequence of arrows with compatible heads and tails. The path algebra of $Q$, denoted $\kk Q$, is a graded associative algebra which has a canonical $\kk$-basis consisting of a monomial for each path in $Q$. We define the composition of two paths as their concatenation if defined and 0 otherwise, and extend linearly over $\kk$ to form the path algebra.

While $\kk Q$ is typically not commutative, if the quiver is finite and has no cycles then $\kk Q$ is finite-dimensional. The indecomposable projective $\kk Q$-modules are summands $P_v = e_v \kk Q$ of $\kk Q$ containing the monomials\footnote{The monomial $e_v\in A_0$ is the idempotent corresponding to a path of length zero at each vertex $v$, which is not considered a loop. In particular, the identity element $1\in A$ is the sum of these orthogonal idempotents.} corresponding to paths starting at a vertex $v$. This implies that maps of projective $\kk Q$-modules can be written as graded matrices with entries in $\kk Q$.

\begin{example}\label{ex:P1xP1-quiver}
  Consider the following quiver with four vertices and eight arrows:
  \[
\begin{tikzcd}
  \bullet & \bullet \\
  \bullet & \bullet
  \arrow["{x_1}"', shift right=1, from=1-1, to=1-2]
  \arrow["{y_1}", shift left=1, from=1-2, to=2-2]
  \arrow["{\bar y_0}"', shift right=1, from=1-1, to=2-1]
  \arrow["{\bar x_1}"', shift right=1, from=2-1, to=2-2]
  \arrow["{x_0}", shift left=1, from=1-1, to=1-2]
  \arrow["{\bar x_0}", shift left=1, from=2-1, to=2-2]
  \arrow["{\bar y_1}", shift left=1, from=1-1, to=2-1]
  \arrow["{y_0}"', shift right=1, from=1-2, to=2-2]
\end{tikzcd}
\]

  As an example, the monomial $x_0y_1\in\kk Q$ corresponds to the path $\bullet \xto{x_0} \bullet \xto{y_1} \bullet$ and the morphisms between the indecomposable projective $\kk Q$-modules for the source and sink vertices can be written as $1\times1$ matrices with coefficients $c_i$ in $\kk$:
  \[ \begin{pmatrix}
    c_0 x_0y_0 + c_1 x_0y_1 + c_2 x_1y_0 + c_3 x_1y_1 +
    c_4 \bar y_0 \bar x_0 + c_5 \bar y_0 \bar x_1 +
    c_6 \bar y_1 \bar x_0 + c_7 \bar y_1 \bar x_1
  \end{pmatrix}. \]
\end{example}

An important result from representation theory states that the derived category of right modules over any finite-dimensional $\kk$-algebra $A$ is equivalent to the derived category of right modules over the path algebra of a \emph{bound} quiver, meaning there is a quiver $Q$ and a set of relations $R\subset\kk Q$ such that $\kk Q/R$ is Morita equivalent to $A$ \cite[\S3]{DW2017}.

Returning to our setting, consider the direct sum of line bundles $\cE = \bigoplus_{\aa=\zero}^\nn \OO(-\aa)$ and its endomorphism algebra $A = \End\left(\cE\right)$. In this situation, $A \cong \kk Q/R$ is the path algebra of a bound quiver with vertices corresponding to summands of $\cE$ and paths corresponding to homomorphisms between them, subject to commutativity relations indicating which compositions of arrows agree as homomorphisms of the line bundles.

\begin{example}\label{ex:P1xP1-bound-quiver}
  On $X = \PP1\times\PP1$ the algebra $\End(\OO_X \oplus \OO_X(-1,0) \oplus \OO_X(0,-1) \oplus \OO_X(-1,-1))$ is the path algebra for the bound quiver $(Q, R)$ consisting of the quiver $Q$ from Example~\ref{ex:P1xP1-quiver} subject to commutativity relations $R$ given by
  \( x_i y_j = \bar y_j \bar x_i, \text{ for } i,j\in\{0,1\}. \)
  The notation $\bar x$ is used to distinguish the pairs of distinct arrows corresponding to multiplication by $x_i$ and $y_j$.  Hence the morphisms between indecomposable projective $A$-modules corresponding to the source vertex $\OO_X(-1,-1)$ and sink vertex $\OO_X$ can be written using only four monomials:
  \[ \begin{pmatrix}
    c_0 x_0y_0 + c_1 x_0y_1 + c_2 x_1y_0 + c_3 x_1y_1
  \end{pmatrix}. \]
\end{example}

In \cite[Thm.~6.2]{Bondal90}, Bondal showed that when the summands of $\cE$ form a full strong exceptional collection, the functor
\begin{equation}\label{eq:RHom}
	\RHom(\cE,-)\colon\Db{X}\to\Db{\mathrm{mod-}A},
\end{equation}
is an equivalence of categories. By a direct computation present in the proof of \cite[Thm.~6.2]{Bondal90}, this functor sends the bundles $\OO(-\aa)$ to projective modules $P_\aa$ for $\zero\leq\aa\leq\nn$. Moreover, minimality of the maps is preserved, in the sense that a non-constant map $\OO(-\aa)\gets\OO(-\bb)$ is sent to the map $e_\aa A\gets e_\bb A$ corresponding to the monomial for the path beginning at vertex $\bb$ and ending at vertex $\aa$. Composing with the reverse equivalence $-\Lotimes\cE$ gives the identity on these objects by the proof of \cite[Thm.~2.1]{King97}.

\end{appendix}

%%%%%%%%%%%%%%%%%%%%%%%%%%%%%%%%%%%%%%%%%%%%%%%%%%%%%%%%%%%%%%%%%%%%%%%%%%%%%%%%%%%%%

%\nocite{*}
%\renewcommand{\section}{\subsection}
%\bibliography{sources}
\printbibliography

@book {DW2017,
  author       = {Derksen, Harm and Weyman, Jerzy},
  title	       = {An introduction to quiver representations},
  series       = {Graduate Studies in Mathematics},
  volume       = {184},
  publisher    = {American Mathematical Society, Providence, RI},
  year	       = {2017},
}

@article{GHL13,
  title	       = {All complete intersection {Calabi--Yau} four-folds},
  author       = {Gray, James and Haupt, Alexander S. and Lukas, Andre},
  journal      = {J. High Energy Phys.},
  fjournal     = {Journal of High Energy Physics},
  year	       = {2013},
  number       = {7},
}

@article{BE24a,
  AUTHOR       = {Brown, Michael K. and Erman, Daniel},
  TITLE	       = {Linear strands of multigraded free resolutions},
  JOURNAL      = {Math. Ann.},
  FJOURNAL     = {Mathematische Annalen},
  VOLUME       = 390,
  YEAR	       = 2024,
  NUMBER       = 2,
  PAGES	       = {2707--2725},
}

@article{BE24b,
  title	       = {Positivity and Nonstandard Graded {{Betti}} Numbers},
  author       = {Brown, Michael K. and Erman, Daniel},
  year	       = 2024,
  journal      = {Bulletin of the London Mathematical Society},
  volume       = 56,
  number       = 1,
  pages	       = {111--123},
}

@article{BE25,
  AUTHOR       = {Brown, Michael K. and Erman, Daniel},
  TITLE	       = {Linear syzygies of curves in weighted projective space},
  JOURNAL      = {Compos. Math.},
  FJOURNAL     = {Compositio Mathematica},
  VOLUME       = 161,
  YEAR	       = 2025,
  NUMBER       = 4,
  PAGES	       = {916--944},
}

@article{BE26,
  author       = {Brown, Michael K. and Erman, Daniel},
  title	       = {Tate resolutions on toric varieties},
  journal      = {J. Eur. Math. Soc. (JEMS)},
  fjournal     = {Journal of the European Mathematical Society (JEMS)},
  volume       = 28,
  year	       = 2026,
  number       = 1,
  pages	       = {269--304},
}

@article{HNVT22,
  AUTHOR       = {Harada, Megumi and Nowroozi, Maryam and Van Tuyl, Adam},
  TITLE	       = {Virtual resolutions of points in {$\mathbb P^1\times\mathbb P^1$}},
  JOURNAL      = {J. Pure Appl. Algebra},
  FJOURNAL     = {Journal of Pure and Applied Algebra},
  VOLUME       = {226},
  YEAR	       = {2022},
  NUMBER       = {12},
  PAGES	       = {Paper No. 107140, 18}
}

@article{Loper21,
  AUTHOR       = {Loper, Michael C.},
  TITLE	       = {What makes a complex a virtual resolution?},
  JOURNAL      = {Trans. Amer. Math. Soc. Ser. B},
  FJOURNAL     = {Transactions of the American Mathematical Society. Series B},
  VOLUME       = {8},
  YEAR	       = {2021},
  PAGES	       = {885--898}
}

@article{Bruce22,
  author       = {Bruce, Juliette},
  title	       = {The quantitative behavior of asymptotic syzygies for {H}irzebruch surfaces},
  journal      = {J. Commut. Algebra},
  fjournal     = {Journal of Commutative Algebra},
  volume       = {14},
  year	       = {2022},
  number       = {1},
  pages	       = {19--26}
}

@article{Bruce24,
  author       = {Bruce, Juliette},
  title	       = {Asymptotic syzygies for products of projective spaces},
  journal      = {J. Algebra},
  fjournal     = {Journal of Algebra},
  volume       = {649},
  year	       = {2024},
  pages	       = {347--391},
}

@article{CN20,
  title	       = {Multigraded regularity of complete intersections},
  author       = {Marc Chardin and Navid Nemati},
  year	       = {2020},
  eprint       = {2012.14899},
  archivePrefix= {arXiv},
  primaryClass = {math.AC}
}

@article{BKLY21,
  title	       = {Homological and Combinatorial Aspects of Virtually {{Cohen}}--{{Macaulay}} Sheaves},
  author       = {Berkesch, Christine and Klein, Patricia and Loper, Michael C. and Yang, Jay},
  year	       = {2021},
  journal      = {Transactions of the London Mathematical Society},
  volume       = {8},
  number       = {1},
  pages	       = {413--434}
}

@article{Yang21,
  author       = {Yang, Jay},
  title	       = {Virtual resolutions of monomial ideals on toric varieties},
  journal      = {Proc. Amer. Math. Soc. Ser. B},
  fjournal     = {Proceedings of the American Mathematical Society. Series B},
  volume       = {8},
  year	       = {2021},
  pages	       = {100--111},
}

@article{Ha07,
  author       = {H\`a, Huy T\`ai},
  title	       = {Multigraded regularity, {$a^*$}-invariant and the minimal free resolution},
  journal      = {J. Algebra},
  fjournal     = {Journal of Algebra},
  volume       = {310},
  year	       = {2007},
  number       = {1},
  pages	       = {156--179},
}

@article{HS07,
  author       = {H\`a, Huy T\`ai and Strunk, Brent},
  title	       = {Minimal free resolutions and asymptotic behavior of multigraded regularity},
  journal      = {J. Algebra},
  fjournal     = {Journal of Algebra},
  volume       = {311},
  year	       = {2007},
  number       = {2},
  pages	       = {492--510},
}

@article{Hering10,
  author       = {Hering, Milena},
  title	       = {Multigraded regularity and the {Koszul} property},
  journal      = {J. Algebra},
  fjournal     = {Journal of Algebra},
  volume       = {323},
  year	       = {2010},
  number       = {4},
  pages	       = {1012--1017},
}

@article{HSS06,
  author       = {Hering, Milena and Schenck, Hal and Smith, Gregory G.},
  title	       = {Syzygies, multigraded regularity and toric varieties},
  journal      = {Compos. Math.},
  fjournal     = {Compositio Mathematica},
  volume       = {142},
  year	       = {2006},
  number       = {6},
  pages        = {1499--1506},
}

@article{CMR07,
  author       = {Costa, L. and Mir\'{o}-Roig, R. M.},
  title	       = {{$m$}-blocks collections and {Castelnuovo}--{Mumford} regularity in multiprojective spaces},
  journal      = {Nagoya Math. J.},
  fjournal     = {Nagoya Mathematical Journal},
  volume       = {186},
  year	       = {2007},
  pages	       = {119--155},
}

@article{HW04,
  author       = {Hoffman, J. William and Wang, Hao Hao},
  title	       = {Castelnuovo--{Mumford} regularity in biprojective spaces},
  journal      = {Adv. Geom.},
  fjournal     = {Advances in Geometry},
  volume       = {4},
  year	       = {2004},
  number       = {4},
  pages	       = {513--536},
}

@article{PP04,
  author       = {Pareschi, Giuseppe and Popa, Mihnea},
  title	       = {Regularity on abelian varieties. {II}. {B}asic results on linear series and defining equations},
  journal      = {J. Algebraic Geom.},
  fjournal     = {Journal of Algebraic Geometry},
  volume       = {13},
  year	       = {2004},
  number       = {1},
  pages	       = {167--193},
}

@article{PP03,
  author       = {Pareschi, Giuseppe and Popa, Mihnea},
  title	       = {Regularity on abelian varieties. {I}},
  journal      = {J. Amer. Math. Soc.},
  fjournal     = {Journal of the American Mathematical Society},
  volume       = {16},
  year	       = {2003},
  number       = {2},
  pages	       = {285--302},
}

@article{EG84,
  author       = {Eisenbud, David and Goto, Shiro},
  title	       = {Linear free resolutions and minimal multiplicity},
  journal      = {Journal of Algebra},
  volume       = {88},
  number       = {1},
  year	       = {1984},
  pages	       = {89--133},
}

@article{BC17,
  author       = {Botbol, Nicol\'{a}s and Chardin, Marc},
  title	       = {Castelnuovo--{Mumford} regularity with respect to multigraded ideals},
  journal      = {Journal of Algebra},
  volume       = {474},
  year	       = {2017},
  pages	       = {361--392},
}

@article{BES20,
  author       = {Berkesch, Christine and Erman, Daniel and Smith, Gregory G.},
  title	       = {Virtual resolutions for a product of projective spaces},
  journal      = {Algebraic Geometry},
  volume       = {7},
  number       = {4},
  year	       = {2020},
  pages	       = {460--481},
}

@article{EFS03,
  author       = {Eisenbud, David and Fl{\o}ystad, Gunnar and Schreyer, Frank-Olaf},
  title	       = {Sheaf cohomology and free resolutions over exterior algebras},
  fjournal     = {Transactions of the American Mathematical Society},
  journal      = {Trans. Amer. Math. Soc.},
  volume       = {355},
  number       = {11},
  year	       = {2003},
  pages	       = {4397--4426},
}

@article{EES15,
  author       = {Eisenbud, David and Erman, Daniel and Schreyer, Frank-Olaf},
  title	       = {Tate resolutions for products of projective spaces},
  fjournal     = {Acta Mathematica Vietnamica},
  journal      = {Acta Math. Vietnam.},
  volume       = {40},
  year	       = {2015},
  number       = {1},
  pages	       = {5--36},
}

@article{Benson04,
  author       = {Benson, Dave},
  title	       = {Dickson invariants, regularity and computation in group cohomology},
  journal      = {Illinois J. Math.},
  fjournal     = {Illinois Journal of Mathematics},
  volume       = {48},
  year	       = {2004},
  number       = {1},
  pages	       = {171--197},
}

@article{MS04,
  author       = {Maclagan, Diane and Smith, Gregory G.},
  title	       = {Multigraded {Castelnuovo}--{Mumford} regularity},
  journal      = {J. Reine Angew. Math.},
  year         = {2004},
  volume       = {571},
  pages	       = {179--212},
}

@article{SVT04,
  author       = {Sidman, Jessica and {Van Tuyl}, Adam},
  title	       = {Multigraded regularity: syzygies and fat points},
  fjournal     = {Beitr\"age zur Algebra und Geometrie},
  journal      = {Beitr. Algebra Geom.},
  volume       = {47},
  number       = {1},
  year	       = {2004},
  pages	       = {67--87},
}

@article{SVTW06,
  author       = {Sidman, Jessica and {Van Tuyl}, Adam and Wang, Haohao},
  title	       = {Multigraded regularity: coarsenings and resolutions},
  fjournal     = {Journal of Algebra},
  journal      = {J. Algebra},
  volume       = {301},
  year	       = {2006},
  number       = {2},
  pages	       = {703--727},
}

@article{Chipalkatti00,
  author       = {Chipalkatti, Jaydeep V.},
  title	       = {A generalization of {C}astelnuovo regularity to {G}rassmann varieties},
  journal      = {Manuscripta Math.},
  fjournal     = {Manuscripta Mathematica},
  volume       = {102},
  year	       = {2000},
  number       = {4},
  pages	       = {447--464},
}

@incollection{BM91,
  author       = {Bayer, Dave and Mumford, David},
  title	       = {What can be computed in algebraic geometry?},
  booktitle    = {Computational algebraic geometry and commutative algebra ({C}ortona, 1991)},
  series       = {Sympos. Math.},
  volume       = {XXXIV},
  publisher    = {Cambridge Univ. Press},
  year	       = {1993},
}

@incollection{Bondal90,
  author       = {Bondal, A. I.},
  title	       = {Helices, representations of quivers and {K}oszul algebras},
  booktitle    = {Helices and vector bundles},
  series       = {London Math. Soc. Lecture Note Ser.},
  volume       = {148},
  pages	       = {75--95},
  publisher    = {Cambridge Univ. Press},
  year	       = {1990},
}

@article{Bondal89,
  author       = {Bondal, A. I.},
  title	       = {Representations of associative algebras and coherent sheaves},
  journal      = {Izv. Akad. Nauk SSSR Ser. Mat.},
  fjournal     = {Izvestiya Akademii Nauk SSSR. Seriya Matematicheskaya},
  volume       = {53},
  year	       = {1989},
  number       = {1},
  pages	       = {25--44},
}

@article{AO89,
  author       = {Ancona, Vincenzo and Ottaviani, Giorgio},
  title	       = {An introduction to the derived categories and the theorem of {Beilinson}},
  fjournal     = {Atti della Accademia Peloritana dei Pericolanti},
  journal      = {Atti Accad. Peloritana Pericolanti},
  volume       = {67},
  year	       = {1989},
  pages	       = {99--110},
}

@unpublished{King97,
  title	       = {Tilting bundles on some rational surfaces},
  author       = {King, Alastair},
  year	       = {1997},
  url	       = {https://people.bath.ac.uk/masadk/papers/tilt.pdf},
  note	       = {Unpublished manuscript}
}

@article{Cox95,
  author       = {Cox, David A.},
  title	       = {The homogeneous coordinate ring of a toric variety},
  journal      = {J. Algebraic Geom.},
  volume       = {4},
  year	       = {1995},
  number       = {1},
  pages	       = {17--50},
}

@incollection{Mumford70,
  author       = {Mumford, David},
  title	       = {Varieties defined by quadratic equations},
  booktitle    = {Questions on {A}lgebraic {V}arieties},
  series       = {Centro Internazionale Matematico Estivo (C.I.M.E.) Summer Schools},
  volume       = {51},
  pages	       = {29--100},
  year	       = {1970},
}

@inproceedings{Kleiman71,
  title = {{Les Theoremes de Finitude pour le Foncteur de Picard}},
  booktitle = {{Th{\'e}orie des Intersections et Th{\'e}or{\`e}me de Riemann-Roch}},
  author = {Kleiman, S.},
  editor = {Berthelot, P. and Jussila, O. and Grothendieck, A. and Raynaud, M. and Kleiman, S. and Illusie, L. and Berthelot, Pierre},
  year = {1971},
  pages = {616--666},
  publisher = {Springer},
}

@book{Weibel1994,
  author       = {Weibel, Charles A.},
  title	       = {An introduction to homological algebra},
  series       = {Cambridge Studies in Advanced Mathematics},
  volume       = 38,
  publisher    = {Cambridge Univ. Press},
  year	       = 1994,
}

@book{Eisenbud2005,
  author       = {Eisenbud, David},
  title	       = {The Geometry of Syzygies},
  series       = {Graduate Texts in Mathematics},
  volume       = {229},
  publisher    = {Springer},
  year	       = {2005}
}

@book{Hartshorne1966,
  author       = {Hartshorne, Robin},
  title	       = {Residues and Duality},
  series       = {Lecture Notes in Mathematics},
  volume       = {20},
  publisher    = {Springer},
  year	       = {1966}
}

@book{Huybrechts2006,
  author       = {Huybrechts, Daniel},
  title	       = {Fourier--{Mukai} Transforms in Algebraic Geometry},
  series       = {Oxford Mathematical Monographs},
  publisher    = {Clarendon Press},
  year	       = {2006}
}

@book{Mumford1966,
  author       = {Mumford, David},
  title	       = {Lectures on Curves on an Algebraic Surface},
  series       = {Ann. of Math. Studies},
  volume       = {59},
  publisher    = {Princeton University Press},
  year	       = {1966}
}

@book{OSS1980,
  title	       = {Vector bundles on complex projective spaces},
  author       = {Okonek, Christian and Schneider, Michael and Spindler, Heinz},
  series       = {Progress in Mathematics},
  volume       = {3},
  publisher    = {Birkhäuser},
  year	       = {1980},
}

@book{CLS2011,
  author       = {Cox, David A. and Little, John B. and Schenck, Henry K.},
  title	       = {Toric varieties},
  series       = {Graduate Studies in Mathematics},
  volume       = {124},
  publisher    = {American Mathematical Society, Providence, RI},
  year	       = {2011},
}

@book{24hours,
  author       = {Iyengar, Srikanth B. and Leuschke, Graham J. and
                  Leykin, Anton and Miller, Claudia and Miller, Ezra
                  and Singh, Anurag K. and Walther, Uli},
  title	       = {Twenty-four hours of local cohomology},
  series       = {Graduate Studies in Mathematics},
  volume       = {87},
  publisher    = {American Mathematical Society, Providence, RI},
  year	       = {2007},
}

@misc{M2,
  label	       = "M2",
  author       = {The {M}acaulay2 project authors},
  title	       = {Macaulay2, a software system for research in algebraic geometry},
  note	       = {available at \url{https://Macaulay2.com}},
}

@article{ABLS20,
  Author       = {Almousa, Ayah and Bruce, Juliette and Loper, Michael C. and Sayrafi, Mahrud},
  Title	       = {The Virtual Resolutions Package for {Macaulay2}},
  journal      = {Journal of Software for Algebra and Geometry},
  volume       = {10},
  number       = {1},
  pages	       = {51--60},
  Year	       = {2020},
}

@article{CHN22,
  AUTHOR       = {Cranton Heller, Lauren and Nemati, Navid},
  TITLE	       = {Linear truncations package for {Macaulay2}},
  JOURNAL      = {J. Softw. Algebra Geom.},
  FJOURNAL     = {Journal of Software for Algebra and Geometry},
  VOLUME       = {12},
  YEAR	       = {2022},
  NUMBER       = {1},
  PAGES	       = {11--16}
}

@article {BM11,
  AUTHOR       = {Ballico, Edoardo and Malaspina, Francesco},
  TITLE	       = {Regularity and cohomological splitting conditions
                  for vector bundles on multiprojective spaces},
  JOURNAL      = {J. Algebra},
  FJOURNAL     = {Journal of Algebra},
  VOLUME       = {345},
  YEAR	       = {2011},
  PAGES	       = {137--149},
}

\end{document}